\documentclass[11pt,a4paper,reqno]{amsart}
\usepackage{amsmath,amsfonts,amsthm,amssymb,graphicx,amscd,color,colortbl}
\usepackage{float,epsf,subfig}
\usepackage[english]{babel}
\usepackage[latin1]{inputenc}
\def\bmx#1{\left(\begin{array}{@{}#1@{}}}
\def\emx{\end{array}\right)}

\def\rr{\mathbf{r}}
\def\ll{\mathbf{l}}
\def\ppsi{{\boldsymbol \psi}}

\newcommand{\mat}[1]{\left[ \begin{array}{#1}}
\newcommand{\rix}{\end{array}\right] }

\theoremstyle{plain}
\newtheorem{coll}{Corollary}[section]
\newtheorem{thm}{Theorem}[section]
\newtheorem{lem}{Lemma}[section]
\newtheorem{prop}{Proposition}[section]
\theoremstyle{definition}

\numberwithin{equation}{section}

\begin{document}
\title[Nonlinear Geometric Optics for Hyperbolic Systems of Conservation Laws]{Weakly Nonlinear Geometric Optics\\
 for Hyperbolic Systems of Conservation Laws}
\author{Gui-Qiang Chen $\qquad$ Wei Xiang $\qquad$ Yongqian Zhang}

\address{Gui-Qiang G. Chen, Mathematical Institute, University of Oxford, 24-27 St Giles, Oxford, OX1 3LB, UK; School of Mathematical Sciences, Fudan University, Shanghai 200433, China}
\email{\tt Gui-Qiang.Chen@maths.ox.ac.uk}

\address{Wei Xiang, School of Mathematical Sciences, Fudan University, Shanghai 200433, China; Mathematical Institute, University of Oxford, 24-27 St Giles, Oxford, OX1 3LB, UK.}
\email{\tt xiang@maths.ox.ac.uk}

\address{Yongqian Zhang,
         Key Laboratory of Mathematics for Nonlinear Sciences and Shanghai Key Laboratory for Contemporary Applied Mathematics,
         School of Mathematical Sciences, Fudan University, Shanghai 200433, China.}
      \email{yongqianz@fudan.edu.cn}
\date{\today}

\maketitle

\begin{abstract}
We present a new approach to analyze the validation of weakly nonlinear geometric optics
for entropy solutions of nonlinear hyperbolic systems of conservation laws whose
eigenvalues are allowed to have constant multiplicity
and corresponding characteristic fields to be linearly
degenerate.
The approach is based on our careful construction of more accurate
auxiliary approximation to
weakly nonlinear geometric optics,
the properties of wave front-tracking approximate solutions,
the behavior of solutions to the approximate asymptotic equations,
and the standard semigroup estimates.
To illustrate this approach more clearly, we focus first on
the Cauchy problem for the hyperbolic systems
with compact support initial data of small bounded variation
and establish that the $L^1-$estimate
between the entropy solution and the geometric optics expansion
function is bounded by $O(\varepsilon^2)$, {\it independent of} the time variable.
This implies that the simpler geometric optics expansion functions
can be employed to study the behavior of general entropy solutions
to hyperbolic systems of conservation laws.
Finally, we extend the results to the case with non-compact support
initial data of bounded variation.
\end{abstract}

\section{Introduction}
We are concerned with weakly nonlinear geometric optics for entropy solutions of the
following hyperbolic system of conservation laws:
\begin{equation}\label{ngo1}
       \begin{cases}
       \partial_tU + \partial_x F(U)=0, \qquad U\in\mathbb{R}^n, \\
       U|_{t=0}=U^0+\varepsilon\, U^1(x),
       \end{cases}
\end{equation}
where $F$: $\mathbb{R}^n\to \mathbb{R}^n$ is a smooth function.
The Jacobian matrix $\nabla F=\nabla_U F(U)$ is diagonalizable at any point $U\in \mathbb{R}^n$
and has $n$ real eigenvalues such that
any eigenvalue has constant multiplicity. Without loss of generality, we assume that
\begin{equation}\label{1.2}
\lambda_1(U)<\cdots<\lambda_m(U)<\lambda_{m+1}(U)\equiv\cdots\equiv\lambda_{m+p}(U)
\stackrel{\triangle}{=}\lambda(U)<\lambda_{m+p+1}(U)<\cdots<\lambda_n(U),
\end{equation}
where
$1\le m\le n-1$, $1\le p\le n-m$,
and the corresponding left and right eigenvectors
$\{\ll_j(U)\}_{j=1}^n$ and $\{\rr_k(U)\}_{k=1}^n$ satisfy
\begin{equation}\label{ngo2}
\ll_j(U)\,\nabla F(U)=\lambda_j \ll_j(U),\;~~\nabla F(U)\, \rr_k(U)=\lambda_k(U)\rr_k(U),\qquad \ll_j(U)\cdot \rr_k(U)=\delta_{jk}.
\end{equation}
Here $\delta_{jk}$ is the Kronecker delta satisfying $\delta_{jj}=1$ and $\delta_{jk}=0$ when $j\ne k$.
When $p=1$, the system in \eqref{ngo1} is strictly hyperbolic.
In addition, we assume that each characteristic field is either genuinely nonlinear
or linearly degenerate, that is, either of the following holds:
\begin{equation}\label{ngo3}
\nabla_U\lambda_j(U)\cdot \rr_j(U)\equiv 1\;
\quad \mbox{or} \quad
\nabla_U\lambda_j(U)\cdot \rr_j(U)\equiv0\qquad\,\,\, \mbox{for all}\;\; U\in \mathbb{R}^n,\,\,
j=1,\cdots, n.
\end{equation}

A prototype of such hyperbolic systems is the full steady Euler equations
for supersonic ideal gases in
$\mathbb{R}^{2}$ with the following form for $\mathbf{x}\in \mathbb{R}^2$:
$$
\left\{
       \begin{array}{l}
       \nabla_{\mathbf{x}}\cdot(\rho\mathbf{v})=0,\\
       \nabla_{\mathbf{x}}\cdot(\rho\mathbf{v}\otimes\mathbf{v})+\nabla p=0,\\
       \nabla_{\mathbf{x}}\cdot((\frac{1}{2}\rho|\mathbf{v}|^2+\frac{\gamma p}{\gamma-1})\mathbf{v})=0,
       \end{array}
       \right.
$$
where $\rho$ is the density, $\mathbf{v}=(v_1,v_2)$ the fluid velocity, and $p$ the pressure.
The notation $\mathbf{a}\otimes\mathbf{b}$ denotes the tensor product of the vectors $\mathbf{a}$ and $\mathbf{b}$.

Set $U=(\rho v_1, \rho v_1 \mathbf{v}, \rho v_1(\frac{1}{2}|\mathbf{v}|^2+\frac{\gamma p}{(\gamma-1)\rho}))$.
When $v_1\ne 0$, the vector
$$
(\rho v_2, \rho v_2 \mathbf{v}, \rho v_2(\frac{1}{2}|\mathbf{v}|^2+\frac{\gamma p}{(\gamma-1)\rho}))
$$
can be expressed as a vector function $F(U)$ of $U$. Then the system can be written into
the form in \eqref{ngo1} with $(t,x)=(x_1, x_2)$.
By direct calculation, we find that the eigenvalues of this system are
$$
\lambda_0=\frac{v_2}{v_1} \,\, (\mbox{repeated}),\qquad \lambda_{\pm}=\frac{v_1v_2\pm c\sqrt{|\mathbf{v}|^2-c^2}}{v_1^2-c^2},
$$
and the corresponding eigenvectors are $\rr_{01}, \rr_{02}$, and $\rr_\pm$.
Then
$$
\begin{array}{l}
\rr_{\pm}\cdot\nabla\lambda_{\pm}\ne0 \,\,\,\qquad \mbox{for all}\;\; |\mathbf{v}|>c,\\[2mm]
\rr_{0j}\cdot\nabla\lambda_0\equiv0 \,\,\,\, \qquad \mbox{for} \;\; j=1, 2.
\end{array}
$$
Thus, the two characteristic fields corresponding to $\lambda_{\pm}$ are genuinely nonlinear, while
the other two fields corresponding to $\lambda_0$ are linearly degenerate.

\medskip
An asymptotic geometric optics expansion is of the following form:
$$
U^{\varepsilon}(t,x)=U^0+\varepsilon\, V^{\varepsilon}(t,x),
$$
where $U^0$ is a constant background state in \eqref{ngo1}. By a formal derivation of the expansion of weakly
nonlinear geometric optics for conservation laws ({\it cf.} DiPerna-Majda \cite{dm-cmp1985313}),
the expansion is expected to be
$$
U(t,x)=U^0+\varepsilon\sum_{j=1}^n\sigma^{(j)}(\varepsilon t,x-\lambda_j(U^0)t)\rr_j(U^0)+o(\varepsilon),
$$
and the functions $\sigma^{(j)}(\tau,y)$ satisfy a decoupled system of scalar conservation laws:
$$
\partial_\tau\sigma^{(j)}+b_j^0\partial_y(\sigma^{(j)})^2=0,\qquad 1\le j\le n,
$$
with the coefficients
$$
b_j^0=\frac{1}{2}\ll_j(U^0)\cdot \nabla^2_U F(U^0)(\rr_j(U^0),\rr_j(U^0))
=\frac{1}{2}\rr_j(U^0)\cdot\nabla_U\lambda(U^0).
$$
Then the genuinely nonlinear condition from \eqref{ngo3} implies that $b_j^0=\frac{1}{2}\neq0$
which yields that equation \eqref{scaler equation}
is the inviscid Burgers equation:
\begin{equation}\label{scaler equation}
\partial_\tau\sigma^{(j)}+ \frac{1}{2}\partial_y(\sigma^{(j)})^2=0,
\end{equation}
while the linearly degenerate condition
implies that $b_j^0=0$ which yields a linear equation
\begin{equation}\label{scalar-equation-2}
\partial_\tau \sigma^{(j)}\equiv0.
\end{equation}
We define
\begin{equation}\label{og expansion}
U_w^{\varepsilon}(t,x)=U^0+\varepsilon\sum_{j=1}^n\sigma^{(j)}(\varepsilon t,x-\lambda_j(U^0)t)\rr_j(U^0).
\end{equation}
Then the much simpler function $U_w^{\varepsilon}$ can be used to study the behavior
of general entropy solutions of nonlinear hyperbolic system of conservation laws,
provided that the convergence can be rigorously verified.

An approach has been introduced in Chen-Christoforou-Zhang \cite{ccz-iumj20072535,ccz-arma200897},
based on the results presented in Bressan \cite{b-book2000},
to compare the solutions of two different systems,
which requires that one of them is the standard Riemann semigroup (SRS) while the other is only a global entropy
solution with bounded variation obtained by the front tracking method.
Let $\mathcal{D}\subset L^1(\mathbb{R}; \mathbb{R}^n)$ be a closed domain.
A map $S: \mathcal{D}\times[0,\infty[\longmapsto\mathcal{D}$ is a {\bf SRS} generated
by system \eqref{ngo1}  if the following three conditions hold ({\it cf.} \cite{b-book2000}):
\begin{description}
\item[$\bullet$ {Semigroup property}] For every $\bar{U}\in \mathcal{D}$ and $t,s\ge 0$,
      \begin{equation}\label{seimgroup property}
       S_0\bar{U}=\bar{U},\qquad S_tS_s\bar{U}=S_{t+s}\bar{U};
       \end{equation}
\item[$\bullet$ Lipschitz continuity] There exist constants $L_1$ and $L_2$ such that,
for all $\bar{U}, \bar{V}\in\mathcal{D}$ and $s,t\ge 0,$
       \begin{equation}\label{lip continuity}
       \|S_t\bar{U}-S_s{\bar{V}}\|_{L^1}\le L_1\|\bar{U}-\bar{V}\|_{L^1}+L_2|t-s|;
       \end{equation}
\item[$\bullet$ Consistency with the Riemann solver] For any piecewise constant initial data $\bar{U}\in\mathcal{D}$,
       there exists $\delta>0$ such that, for all $t\in[0,\delta]$, the trajectory $U(t,\cdot)=S_t\bar{U}(\cdot)$ coincides with
       the solution of the Cauchy problem \eqref{ngo1} obtained by piecing together the standard solutions for the
       Riemann problems determined by the jumps of $\bar{U}$.
\end{description}

Following \cite{ccz-iumj20072535,ccz-arma200897},
in this paper, the general entropy solution $U(t,x)$ of the Cauchy problem \eqref{ngo1} under consideration is the SRS,
which can be constructed by the front tracking method ({\it cf.} \cite{ky-aml2003143}),
and the entropy solution for the corresponding asymptotic scalar equation \eqref{scaler equation}
is constructed by polygonal approximations, first introduced in Dafermos \cite{d-jmaa197233}, with initial data:
\begin{equation}\label{scaler initial data}
\sigma^{(j)}(0,y)
=\ll_j(U^0)\cdot U^1(y),\qquad 1\le j\le n.
\end{equation}
We establish the $L^1$-estimate between $U(t,x)$ and $U_w^{\varepsilon}(t,x)$ by using both the properties
of the wave-front tracking algorithm and
the standard error formula ({\it cf.} \cite{b-book2000}):
\begin{equation}\label{semigroup estimate}
\|S_TW(0)-W(T)\|_{L^1}\le L\int_0^T\liminf_{h\rightarrow 0+}\frac{\|S_hW(\tau)-W(\tau+h)\|_{L^1}}{h}\mbox{d}\tau,
\end{equation}
where $L$ is the Lipschitz constant of the semigroup $S_t$,
and $W(t)$ is any Lipschitz continuous map defined on $[0,T]$.
One of our objectives here is to develop a new approach to provide a
rigorous mathematical proof of the following theorem.

\begin{thm}\label{thm:compacted initial data}{\bf (Main Theorem).}
Let $F(U)\in C^2(\mathbb{R}^n;\mathbb{R}^n)$, and let $U^1(x)$ be an arbitrary function
of bounded variation with compact support.
Assume that each eigenvalue of the hyperbolic system in \eqref{ngo1} has
constant multiplicity and its corresponding
characteristic field is either genuinely nonlinear or linearly degenerate.
Consider an entropy solution $U^{\varepsilon}(t,x)$ of the Cauchy problem \eqref{ngo1}, which is a SRS, and
the weakly nonlinear geometric optics expansion function
$U^{\varepsilon}_w(t,x)$ defined by
\eqref{scaler equation}--\eqref{og expansion} and \eqref{scaler initial data}.
Then there exists $\varepsilon_0>0$
such that, for all $t>0$ and $\varepsilon\in(0,\varepsilon_0]$,
\begin{equation}\label{compact}
\|U^{\varepsilon}(t,\cdot)-U^{\varepsilon}_w(t,\cdot)\|_{L^1}\le C\, \varepsilon^2,
\end{equation}
for some $C>0$ independent of $\varepsilon$ and $t$.
\end{thm}

We remark here that this result allows the eigenvalues of
the $n\times n$ hyperbolic system in \eqref{ngo1}
to have constant multiplicity and the corresponding characteristic fields
to be linearly degenerate,
which answers the open problem posed by Majda in \cite{m}.
In particular, for the $n\times n$ system in \eqref{ngo1},
we obtain that the $L^1-$estimate
between the entropy solution and the geometric optics expansion
function is bounded by $O(\varepsilon^2)$ that is {\it independent of $t\in [0, \infty)$}.

\smallskip
The proof of Theorem \ref{thm:compacted initial data} is based on our following observation:
For a genuinely nonlinear system with initial data of compact support, the waves of different
families in the solution will be separated each other. This enables us to follow
Majda-Rosaales \cite{mr-sam1984149} and use the front tracking method.
We introduce an auxiliary approximate function $V_w^{\epsilon}$
by adding higher order term of $U_w^\epsilon$ to \eqref{og expansion} as
\begin{equation}\label{approximate solution compact support}
V^{\varepsilon}_w=U^0+\varepsilon\sum_{j=1}^n\sigma^{(j)}(\varepsilon t, x-\lambda_j(U^0)t){\bf r}_j(U^0) +\frac{\varepsilon^2}{2}\sum_{j}^n\big(\sigma^{(j)}(\varepsilon t, x-\lambda_j(U^0)t)\big)^2({\bf r}_j(U^0)\cdot\nabla){\bf r}_{j}(U^0)
\end{equation}
and its corresponding more accurate auxiliary approximate function:
\begin{eqnarray}
V^{\varepsilon}_{\nu}
&=& U^0+\varepsilon\sum_{j}\sigma^{(j)}_{\nu}(\varepsilon t, x-\lambda_j(U^0)t)\rr_j(U^0) \nonumber\\
&& +\frac{\varepsilon^2}{2}\sum_{j\in N}
\big(\sigma^{(j)}_{\nu}(\varepsilon t, x-\lambda_j(U^0)t)\big)^2(\rr_j(U^0)\cdot\nabla)\rr_{j}(U^0)\nonumber \\
&& +\varepsilon^2\sum_{j_i\in L}E_{\nu}^{(j)}(x-\lambda_j(U^0)(t-T_0)), \label{auxi function}
\end{eqnarray}
where $j\in N$ means that the corresponding $j-$th characteristic field is genuinely nonlinear,
while $j_i\in L$ means that the corresponding characteristic field is linearly degenerate
and all $\{j_i\}$ together constitute the $j$-th characteristic field; Furthermore,
$\sigma^{(j)}_{\nu}$ is given in \S \ref{sec:scalar scheme} and $E^{(j)}_{\nu}$ is defined in \S 5.
The novelty here is that the new correction terms are introduced to deal with the
contact discontinuities.
With this key observation, then our approach is to
prove the $L^1$--distance between this auxiliary function
and the general entropy solution to system \eqref{ngo1} with the same initial data is $O(1)\varepsilon^2$,
and finally to employ the $L^1$--stability of solutions with respect to initial
data
to establish Theorem \ref{thm:compacted initial data}.
The complete proof of Theorem \ref{thm:compacted initial data}
will be given in \S \ref{sec:main thm} {and \S \ref{sec:compact support}.
This provides an alternative approach to deal with nonlinear geometric optics
for hyperbolic systems of conservation laws in \eqref{ngo1}.
}

As an example of further applications of this approach,
we extend the result to the case when the initial data has non-compact support.

\begin{thm}\label{thm:noncompact case}
Let $F(U)\in C^2(\mathbb{R}^n;\mathbb{R}^n)$
and $U^1(x)\in BV(\mathbb{R};\mathbb{R}^n)\cap L^{1}(\mathbb{R};\mathbb{R}^n)$.
Assume that each real eigenvalue of the hyperbolic system in \eqref{ngo1} has constant multiplicity and its corresponding
characteristic field is either genuinely
nonlinear or linearly degenerate.
Consider an entropy solution $U^{\varepsilon}(t,x)$ of the Cauchy problem \eqref{ngo1}, which is the SRS, and
the weakly nonlinear geometric optics expansion $U^{\varepsilon}_w(t,x)$ defined by
\eqref{og expansion}.
Then
\begin{equation}\label{ieq:noncompact case}
\sup_{0\leq t\leq T_0/\varepsilon}\|U^{\varepsilon}(t,\cdot)-U^{\varepsilon}_w(t,\cdot)\|_{L^1}=o(\varepsilon)
\qquad \mbox{when}\,\, \varepsilon\rightarrow0.
\end{equation}
\end{thm}

For related earlier results in this direction,
we refer the reader to DiPerna-Majda \cite{dm-cmp1985313} for an order of $O(\varepsilon t^2)$ for
the case of periodic initial data and
{the same estimate \eqref{compact}}
for initial data with compact support for $2\times2$ {genuinely nonlinear}
and strictly hyperbolic systems of conservation laws.
For general strictly hyperbolic systems with some kind of periodic properties
of the initial data for which the resonance phenomena occur, {Schochet \cite{s-jde1994473}
proved the $L^1$-estimate of order $o(\varepsilon^2t)$; and
Cheverry \cite{c-dmj1997213} dealt with more general initial data and proved that, for all $t\ge 0$,
$$
\|U_{\varepsilon}(t,\cdot)-m_{\varepsilon}(t,\cdot)\|_{L^1(K)}=o(\varepsilon)
\qquad\mbox{for any fixed compact set}\,\, K\Subset \mathbb{R},
$$
where $m_{\varepsilon}$ is the corresponding geometric optics expansion.
Both of their results allow the characteristic fields to be linearly degenerate,
but require the hyperbolic system in \eqref{ngo1} to be strictly hyperbolic.}
We also refer the reader to  Chen-Junca-Rascke \cite{cjr-jde2006439},
Cheverry \cite{c-cpde19961119}, Gu\`es \cite{g-dmj1992401}, Hunter-Majda-Rosales \cite{hmr-spm1986187},
Joly-M\'{e}tivier-Rauch \cite{jmr-jfa1993106},
Majda-Rosales \cite{mr-sam1984149}, and the references cited therein for related results.
For classical results on the front tracking method and hyperbolic systems of conservation laws,
see Bressan \cite{b-book2000} and Dafermos \cite{d-book2000}.

\section{Front Tracking Schemes and Standard Riemann Semigroups} \label{sec:preliminary}

In this section, we analyze entropy solutions of hyperbolic systems of conservation laws
and front tracking algorithms for scalar equations for subsequent development.

\subsection{Existence and Stability of Entropy Solutions}

Consider the Riemann problem of system \eqref{ngo1} with the following initial data:
\begin{equation}\label{initial data}
U(0,x)=
\left\{
\begin{array}{ll}
U_+,\qquad x>0, \\[2mm]
U_-,\qquad x<0,
\end{array}
\right.
\end{equation}
where $U_{\pm}$ are constant vectors. Based on the results in \cite{ky-aml2003143}, we have

\begin{lem}\label{prop:Riemann systems}
Assume that $F(U)$ satisfies the same assumptions as in Theorem {\rm \ref{thm:compacted initial data}}.
Then, for every compact set $K\Subset\Omega$, there exists $\delta>0$ such that,
whenever $U_-\in K, |U_+-U_-|\le\delta$, the Riemann problem above has a unique entropy solution,
which consists of $n-p+2$ constant states, denoted by $U_i\;(i=0,1,\cdots,m,m+p,\cdots,n)$,
and $n-p+1$ elementary waves (shock or rarefaction waves corresponding to the genuinely nonlinear
characteristic fields,
or contact discontinuities to the linear degenerate fields).
Moreover, there exists a unique small parameter vector $(\beta_1,\cdots, \beta_n)$ such that
\begin{equation}\label{laxtypemapdf1}
\begin{array}{ll}
U_0=U_-,\quad U_n=U_+=\Psi(U_-;\; \beta_1,\cdots, \beta_n),\\[2mm]
U_i=\ppsi_i(U_{i-1}; \beta_i),\qquad i=0,1,\cdots,m,m+p+1,\cdots,n,\\
\end{array}
\end{equation}
and
\begin{equation}\label{laxtypemapdf2}
U_{m+p}=\ppsi(U_m;\; \beta_{m+1},\cdots, \beta_{m+p}),
\end{equation}
where $\ppsi_i, \ppsi$, and $\Psi$ are smooth functions with respect to the respective parameter vectors and satisfy
\begin{eqnarray}\label{laxtypemapprop}
&&\ppsi_i(U_-;0)=U_-,\quad\frac{\partial\ppsi_i}{\partial \beta_i}(U_-;0)=\rr_i(U_-),\qquad i=0,1,\cdots,m,m+p+1,\cdots,n, \notag\\
&&\frac{\partial\ppsi}{\partial \beta_j}\in \mbox{\rm Ker}\big(\lambda(U)I-\nabla F(U)\big),\qquad j=m+1,\cdots,m+p,\\
&&\ppsi(U_-;0,\cdots,0)=U_-,\,\,\, \frac{\partial\ppsi}{\partial \beta_j}(U_-;0,\cdots,0)=\rr_j(U_-),\quad
j=m+1,\cdots,m+p. \notag
\end{eqnarray}
\end{lem}

From Lemma \ref{prop:Riemann systems}, one can follow the approach in \cite{b-book2000} and \cite{s-book1983}
to obtain
the existence and stability of the unique entropy solution that is the standard Riemann semigroup (SRS).

\begin{lem}[Existence and Stability of SRS]\label{prop:stability}
Assume that $F(U)$ satisfies the same assumptions as in Theorem {\rm \ref{thm:compacted initial data}}.
Then there is a suitably small $\delta_0>0$ such that, given any
$\bar{U}\in L^1(\mathbb{R};\,\mathbb{R}^n)$ with $TV(\bar{U})<\delta_0$,
there exists an entropy solution $U(t,x)=S_t(x)$ by the wave-front tracking method or the Glimm scheme.
The map $S: [0,\infty[\times\mathcal{D}\longmapsto\mathcal{D}$ satisfies that,
for all $\bar{U},\bar{V}\in\mathcal{D}, s, t\ge 0$,
\begin{eqnarray}
&&S_0\bar{U}=\bar{U},\quad S_tS_s\bar{U}=S_{t+s}\bar{U},\\[2mm]
&& \|S_t\bar{U}-S_s\bar{V}\|_{L^1}\le L_1\,\|\bar{U}-\bar{V}\|_{L^1}+L_2\,|t-s|
\,\quad \mbox{for some constants $L_1$ and $L_2$},\label{semigroup property}
\end{eqnarray}
such that the solution is a SRS.
Furthermore, for any other SRS $\, \tilde{S}:\tilde{\mathcal{D}}\times[0,\infty[\longmapsto\tilde{\mathcal{D}}$,
defined on a domain $\tilde{\mathcal{D}}\supset\mathcal{D}$, we have
$$
\tilde{S}_t\bar{U}=S_t\bar{U} \qquad \mbox{for any}\,\, ~\bar{U}\in\mathcal{D},\;t\ge 0.
$$
\end{lem}

In Lemma \ref{prop:stability}, the domain $\mathcal{D}$ is defined by the Glimm functional;
see the definition in \cite{b-book2000} (Chapter 7, pp. 151).

\subsection{The Wave-Front Tracking Scheme for Scaler Equations}\label{sec:scalar scheme}

We adopt the following scheme and related results from \cite{b-book2000}; also see \cite{d-jmaa197233} for the details.
For our problem, we assume that the scalar flux function $f$ is a convex function.
For fixed integer $\nu\ge 1$, we consider the piecewise constant initial data $\bar{u}$ taking values
within the discrete set $2^{-\nu}\mathbb{Z}\doteq \{2^{-\nu}j\,:\, \mbox{j integer} \}$
and define $f_\nu$ to be the piecewise affine function that coincides with $f$
at all nodes $2^{-\nu}j$ with $j$ integer:
$$
f_{\nu}(s)=\frac{s-2^{-\nu}}{2^{-\nu}} f(2^{-\nu}(j+1))+\frac{2^{-\nu}(j+1)-s}{2^{-\nu}}f(2^{-\nu}j),
\quad\; s\in[2^{-\nu}j,\;2^{-\nu}(j+1)].
$$

For the scheme, we consider the Cauchy problem:
\begin{equation}\label{scalar scheme equation}
\partial_t u+ \partial_x f_{\nu}(u)=0
\end{equation}
with initial data $\bar{u}$.
First we consider the Riemann problem with initial data $\bar{u}$ as in \eqref{initial data} and
$u_{\pm}\in2^{-\nu}\mathbb{Z}$ to obtain the following solutions:
\begin{description}
\item[$\bullet$ Case 1. $u_-<u_+$] We define the increasing sequence of jump speeds as
$$
\lambda_{l}=\frac{f_{\nu}(w_l)-f_{\nu}(w_{l-1})}{w_l-w_{l-1}}, \qquad l=1,\cdots,q,
$$
where $\{w_i\}_{i=1}^q$ are the jump points of $f_{\nu}$ and satisfy $w_0\doteq u_-<w_1<\cdots<w_q\doteq u_+$.
Then we can obtain an entropy solution of the above Riemann problem
\eqref{scalar scheme equation} and \eqref{initial data} as follows:
\begin{equation}\label{rarefaction solution for scheme}
w(t,x):= \left\{\begin{array}{l}
                 u_-\qquad \mbox{if}~x<\lambda_1 t,\\
                 w_l\qquad\, \mbox{if}~\lambda_lt <x<\lambda_{l+1} t,\; 1\le l\le q-1,\\
                 u_+\qquad \mbox{if}~ x>\lambda_q t.
\end{array}\right.
\end{equation}

\item[$\bullet$ Case 2. $u_->u_+$] In this case, we define the shock speed as
$$
\lambda=\frac{f_{\nu}(u_+)-f_{\nu}(u_-)}{u_+-u_-}.
$$
Thus, we can also obtain an entropy solution as the previous case:
\begin{equation}\label{shock solution for scheme}
w(t,x):= \left\{\begin{array}{l}
                 u_- \qquad \mbox{if}~x<\lambda t,\\
                 u_+ \qquad \mbox{if}~ x>\lambda t.
\end{array}\right.
\end{equation}
\end{description}
Next, consider a more general Cauchy problem for \eqref{scalar scheme equation}
with piecewise constant initial data $\bar{u}$, taking values within the set $2^{-\nu}\mathbb{Z}$.
We can construct the solution by solving the corresponding Riemann problems so that
the total number of interactions is finite and the solution can be prolonged for all $t\ge 0$.
For these solutions $u_\nu(t,x)$,
we have the following properties:
\begin{eqnarray}\label{scheme scalar solution property}
&&TV(u_{\nu}(t,\cdot))\le TV(\bar{u}),\quad \|u_{\nu}(t,\cdot)\|_{L^{\infty}}\le||\bar{u}||_{L^{\infty}} \,
   \qquad\quad \mbox{for any} \,\, t\ge 0,\\[2mm]
&&\|u_{\nu}(t,\cdot)-u_{\nu}(t',\cdot)\|_{L^1}\le L\, TV(\bar{u})\, |t-t'| \,\quad \qquad \mbox{for  any}\,\, t,t'\ge 0,
\end{eqnarray}
where $L$ is the Lipschitz constant such that
$$
|f(w)-f(w')|\le L|w-w'|\qquad \mbox{for any}\,\, w,w'\in[-M,M].
$$
Then, using Helly's theorem, we obtain an entropy
solution $u=u(t,x)$ defined for all $t\geq 0$ by compactness, with
\begin{equation}\label{scalar solution property}
TV(u(t,\cdot))\le TV(\bar{u}),\quad \|u(t,\cdot)\|_{L^{\infty}}\le \|\bar{u}\|_{L^{\infty}}
\qquad \quad \mbox{for any}\,\, t\ge 0.
\end{equation}

We remark here that, since the scalar flux function in our case in \S 3-- \S 6 is quadratic,
the front-tracking algorithm could be applied directly.
However, we adopt the piecewise-linear approximation in our analysis
in \S 3--\S 6 so that it is more convenient to compare the solutions to present our approach.

\section{Comparison of the Riemann Solvers}\label{sec:riemann comparison}

In this section we compare the Riemann solvers to system \eqref{ngo1}
and the geometric optics expansion defined by \eqref{og expansion} with the same initial data.
From now on, we denote $\lambda_j^0=\lambda_j(U^0)$ and $\rr_j^0=\rr_j(U^0)$ through the paper.

\subsection{Comparison for the Genuinely Nonlinear Case}

First, we consider some properties of approximate solutions to the Burgers equation.

\begin{lem}\label{lem:burger's riemann problem}
For some $1\le k\le n$, assume that, for fixed $\nu$, $\sigma_{\nu}^{(k)}(\tau,y)$
is an approximate solution to the inviscid Burgers equation:
$$
\partial_{\tau}w +\frac{1}{2}\partial_y (w^2)=0,
$$
constructed by the front tracking method.
At the jump point $(\tau_0,y_0)$, suppose that
$\sigma^{(k)}_{\nu -}=\sigma^{(k)}_{\nu}(\tau_0,y_{0}-),~\sigma^{(k)}_{\nu +}=\sigma^{(k)}_{\nu}(\tau_0,y_{0}+)$,
$\sigma^{(k)}_{\nu +}=\sigma^{(k)}_{\nu -}+\sigma$,
and $(t_0,x_0)=(\frac{\tau_0}{\varepsilon}, y+\lambda_k^0\frac{\tau_0}{\varepsilon})$
is the corresponding point in the $(t,x)$-coordinates
to the point $(\tau_0,y_0)$.
Then
\begin{enumerate}
\item[\rm (i)] If $\sigma<0$, the slope of the discontinuity line {of solutions
to the inviscid Burgers equation} in the $(t,x)$--coordinates is
       \begin{equation}\label{slope of scalar shock }
\dot{S}_{B(t,x)}(\sigma)=\lambda_k^0+\sigma_{\nu-}^{(k)}\varepsilon  +\frac{\sigma}{2}\varepsilon;
       \end{equation}
\item[\rm (ii)] If $\sigma>0$, then $\sigma=(\sigma_{\nu+}^{(k)}-\sigma_{\nu-}^{(k)})\, 2^{-\nu}$,
        and the slope of the discontinuity line {of solutions to
the inviscid Burgers equation} in the $(t,x)$--coordinates is
       \begin{equation}\label{slope of scalar rarefaction }
       \lambda^{(k)}_{B(t,x)}(\sigma)=\lambda_k^0+\sigma_{\nu-}^{(k)}\varepsilon +\frac{\sigma}{2}\varepsilon.
       \end{equation}
\end{enumerate}
\end{lem}

\begin{proof} First, we deduce from \S \ref{sec:scalar scheme}
that {the slope of the discontinuity line of solutions to
the inviscid Burgers equation in the $(\tau,y)$--coordinates is}
$$
y'(\tau)=\sigma_{\nu-}^{(k)}+\frac{\sigma}{2}
$$
from the Rankine-Hugoniot condition when $\sigma<0$, and
$$
y'(\tau)=\sigma_{\nu-}^{(k)}+\frac{\sigma}{2}
$$
from
\eqref{rarefaction solution for scheme} when $\sigma>0$.
Using the fact that, if $(t,x)$ is a point on the discontinuity
line with the corresponding point $(\tau,y)$ in the $(\tau,y)$-coordinates, we have
$$
\left\{
\begin{array}{l}
y=x-\lambda_k^0 t,\\[2mm]
\tau=\varepsilon t,
\end{array}
\right. \qquad \mbox{and} \qquad
\left\{
\begin{array}{l}
\dot{S}_{B(t,x)}(\sigma)=\frac{x-x_0}{t-t_0},\quad \sigma <0,\\[2mm]
\lambda_{B(t,x)}^{(k)}(\sigma)=\frac{x-x_0}{t-t_0},\quad \sigma >0.
\end{array}
\right.
$$
This completes the proof.
\end{proof}

Next, we consider the Riemann problem to system \eqref{ngo1} with the corresponding initial data:
\begin{equation}\label{systems corresponding initial data}
U|_{_{t=t_0}}=U^0+\varepsilon \sum_{j=1}^n\sigma^{(j)}_{\nu}(\varepsilon t_0,x-\lambda_j^0\,t_0)\,\rr_j^0.
\end{equation}

\begin{lem}\label{lem:genius nonlinear riemann problem}
Assume that $F(U)$ satisfies all the assumptions as in Theorem {\rm \ref{thm:compacted initial data}},
the $k$-th characteristic field is genuinely nonlinear,
and $\sigma_{\nu}^{(k)}$ has a jump point as {\rm Lemma \ref{lem:burger's riemann problem}}.
Let $(t_0,x_0)$ be the corresponding point to $(\tau_0,y_0)$,
$U^{\varepsilon}_{\pm}=U^0+\varepsilon\sum_{j=1}^n\sigma_{\nu_{\pm}}^{(j)}\rr_j^0$, and
$U^{\varepsilon}_+=\Phi(U_-^{\varepsilon};\beta_1,\cdots,\beta_n)$, where
$\sigma^{(j)}_{\nu+}=\sigma^{(j)}_{\nu-}$ are constant valued functions near $(\tau_0,y_0)$ when $j\neq k$.
Then we have
$$
\beta_j(\sigma,\varepsilon)
=\sigma\varepsilon\delta_{jk} +O(1)\sigma(\max_j|\sigma_{\nu-}^{(j)}|+\sigma)\,\varepsilon^2,
$$
where $\delta_{jk}$ is the Kronecker delta.
\end{lem}

\begin{proof}
Let $\theta=\varepsilon\sigma$. From the definition,
$U^{\varepsilon}_+=U^{\varepsilon}_-+\theta\rr_k^0$.
Then
$$
\Phi(U_-^{\varepsilon}; \beta_1,\cdots,\beta_n)-U_-^{\varepsilon}=\theta\rr_k^0.
$$
If $\theta=0$, we deduce from $\Phi(U_-^{\varepsilon};\beta_1,\cdots,\beta_n)=U_-^\varepsilon$
that $\beta_j|_{\theta=0}=0$.
We differentiate both sides of the above equation with respect to $\theta$ to obtain
$$
\sum_j\frac{\partial\Phi}{\partial\beta_j}\frac{\partial\beta_j}{\partial\theta}=\rr^0_k,
$$
From Lemma \ref{prop:Riemann systems}, we know
$$
\left.\frac{\partial\Phi}{\partial\beta_j}\right|_{\theta=0}
=\left.\frac{\partial\Phi}{\partial\beta_j}\right|_{\beta=0}=\rr_k(U_-^{\varepsilon}).
$$
Then we find
\begin{eqnarray}
&&\left.\sum_j\frac{\partial\Phi}{\partial\beta_j}\right|_{\theta=0}\left.
   \frac{\partial\beta_j}{\partial\theta}\right|_{\theta=0}\nonumber\\
&&=\rr_k^0\nonumber\\
&&=\rr_k(U^{\varepsilon}_- -\varepsilon\sum_{j=1}^n\sigma_{\nu-}^{(j)}\rr_j^0)\nonumber\\
&&=\rr_k(U_-^{\varepsilon})
   -\varepsilon\sum_{j=1}^n(\rr_j^0\cdot\nabla)\rr_k(U_-^{\varepsilon})\sigma_{\nu-}^{(j)}+O(1)\varepsilon^2
\nonumber\\
&&=\rr_k(U_-^{\varepsilon})+O(1)\max_j|\sigma_{\nu-}^{(j)}|\,\varepsilon. \label{3.4a}
\end{eqnarray}

Taking the dot product both sides of \eqref{3.4a} with $\ll_j(U_-^{\varepsilon})$, we obtain
$$
\left.\frac{\partial\beta_j}{\partial\theta}\right|_{\theta=0}
=\delta_{jk}+O(1)\max_j|\sigma_{\nu-}^{(j)}|\,\varepsilon.
$$
Then we have
\begin{eqnarray*}
\beta_j(\theta,\;\varepsilon)=\delta_{jk}\theta
+O(1)\max_j|\sigma_{\nu-}^{(j)}|\theta \varepsilon+O(1)\theta^2\\
=\delta_{jk}\sigma\varepsilon+O(1)\sigma\big(\max_j|\sigma_{\nu-}^{(j)}|+\sigma\big)\varepsilon^2.
\end{eqnarray*}
This completes the proof.
\end{proof}

With Lemmas \ref{lem:burger's riemann problem}--\ref{lem:genius nonlinear riemann problem},
we have

\begin{prop}\label{lem:local L1 estimates for genius nonlinear}
Assume that $F(U)$ satisfies the assumptions as in Theorem {\rm \ref{thm:compacted initial data}},
the $k$-th characteristic field is genuinely nonlinear,
$\sigma_{\nu}^{(k)}$ is a piecewise constant function
as  Lemma {\rm \ref{lem:burger's riemann problem}} and has a jump point at $(\tau_0,y_0)$,
$(t_0,x_0)$ is the corresponding point,
and $(\tau_0,y_0)$ is not a jump point of $\sigma_{\nu}^{(j)}$ for $j\ne k$.
Then, for every $\hat{\lambda}>2\max_{|U|\le M}\lambda_k(U)$ with $M$ being the maximum of
the solution for \eqref{ngo1}, if $h$ is sufficiently small, we have
\begin{equation}\label{3.5a}
\int_{x_0-\hat{\lambda}h<x<x_0+\hat{\lambda}h}
|S_h(U_{w,\nu}^{\varepsilon}(t_0,\cdot))-U_{w,\nu}^{\varepsilon}(t_0+h,x)|
\mbox{d}x
\le C\sigma(\sigma+\max_j|\sigma_{\nu-}^{(j)}|)h \varepsilon^2,
\end{equation}
where
$$
U_w^{\varepsilon}(t,x):=U^0+\varepsilon\sum_{j=1}^n\sigma_{\nu}^{(j)}(\varepsilon t,x-\lambda_j^0\,t)\,\rr_j^0,
$$
and $C>0$ is a constant independent of $\varepsilon$ and $t$.
\end{prop}

\begin{proof}
First, we consider the case $\sigma<0$. Then, from Lemma \ref{lem:genius nonlinear riemann problem},
we obtain that $\beta_k<0$ when $\varepsilon$ sufficiently small. Thus, the $k$-th wave is a shock.
From Lemma \ref{prop:Riemann systems} and the Rankine-Hugoniot condition, we have
$$
\left.\dot{S}(\beta_k)\right|_{\beta_k=0}=\lambda_k(U_{k-1}),\quad
\left.\frac{\partial\dot{S}(\beta_k)}{\partial\beta_k}\right|_{\beta_k=0}
=\frac{1}{2}(\rr_k\cdot\nabla)\lambda_k(U_{k-1}),
$$
where $\dot{S}(\beta_k)$ is the shock speed of the $k$-th shock with respect to $\beta_k$,
and $U_{k-1}$ is defined as in Lemma \ref{prop:Riemann systems}.
Therefore, if $U_-=U^{\epsilon}_{w,\nu-}$ and $U_+=U^{\epsilon}_{w,\nu+}$, we have
$$
\begin{array}{l}
\dot{S}(\beta_k)=\lambda_k(U_{k-1})+\frac{1}{2}\beta_k+O(1)\beta^2 \\[2mm]
\quad\quad\,\,\,\,\, =\lambda_k(U_-)+\frac{1}{2}\beta_k+O(1)\beta^2
+O(1)\sigma(\sigma+\max_j|\sigma_{\nu_-}^{(j)}|)\,\varepsilon^2,
\end{array}
$$
by using Lemmas \ref{prop:Riemann systems} and \ref{lem:genius nonlinear riemann problem}
and the following estimate:
\begin{equation}\label{estimate for uk-1 and u-}
|U_{k-1}-U_-|\le\sum_{j<k}|U_j-U_{j-1}|\le C \sum_{j<k}|\beta_j|
=C\sigma(\max_j|\sigma_{\nu-}^{(j)}|+\sigma)\,\varepsilon^2.
\end{equation}
Using Lemma \ref{lem:genius nonlinear riemann problem}, we have
\begin{eqnarray*}
\dot{S}(\beta_k)&=&\lambda_k(U_-)+\frac{1}{2}\sigma\varepsilon
+O(1)\sigma\big(\sigma+\max_j|\sigma_{\nu-}^{(j)}|\big)\varepsilon^2\\
&=&\lambda_k(U_-)+\frac{1}{2}\sigma\varepsilon
+O(1)\sigma\big(\sigma+\max_j|\sigma_{\nu-}^{(j)}|\big)\varepsilon^2,
\end{eqnarray*}
Then we have
$$
\begin{array}{l}
\dot{S}(\beta_k)-\dot{S}_{B(t,x)}(\sigma)\\[2mm]
=\lambda_k(U_-)+\frac{1}{2}\sigma\varepsilon
+O(1)\sigma\big(\sigma+\max_j|\sigma_{\nu-}^{(j)}|\big)\varepsilon^2
-\big(\lambda_k^0+\sigma_{\nu-}^{(k)}\varepsilon +\frac{\sigma}{2}\varepsilon\big)\\[2mm]
=\lambda_k(U_-)-\lambda_k^0
-\frac{1}{2}\sigma_{\nu-}^{(k)}\varepsilon  +O(1)\sigma\big(\sigma+\max_j|\sigma_{\nu-}^{(j)}|\big)\varepsilon^2\\[2mm]
=\varepsilon\sum_{j\ne k}\sigma_{\nu-}^{(j)}
+O(1)\sigma\max_j\big(|\sigma_{\nu-}^{(j)}|^2\big)\varepsilon^2
+O(1)\sigma\big(\sigma+\max_j|\sigma_{\nu-}^{(j)}|\big)\varepsilon^2\\[2mm]
=O(1)\max_j|\sigma_{\nu-}^{(j)}|\, \varepsilon.
\end{array}
$$
As shown in Figure \ref{fig:nonlinear shock}, with estimate \eqref{estimate for uk-1 and u-}, and
\begin{equation}\label{estimate for uk and u+}
|U_{k}-U_+|\le\sum_{j\ge k}|U_j-U_{j+1}|\le C\sum_{j>k}|\beta_j|
=C \sigma\big(\max_j|\sigma_{\nu-}^{(j)}|+\sigma\big)\varepsilon^2,
\end{equation}
we have
\begin{eqnarray*}
&&\int_{x_0-\hat{\lambda}h<x<x_0+\hat{\lambda}h}
|S_h(U_{w,\nu}^{\varepsilon}(t_0,\cdot))-U_{w,\nu}^{\varepsilon}(t_0+h,x)|\mbox{d}x\\
&&=\int_{x_0-\hat{\lambda}h<x<x_0+\hat{\lambda}h}
|U(t_0+h,x)-U^0-\sum_{j=0}^n\sigma_{\nu}^{(j)}(\varepsilon(t_0+h),x-\lambda_j^0(t_0+h))\rr_j^0|\mbox{d}x\\
&&=\Big(\int_{x_0-\hat{\lambda}h}^{{\min\{\dot{S}_k,\dot{S}_{B(t,x)}\}h}}
+\int_{\max\{\dot{S}_k,\dot{S}_{B(t,x)}\}h}^{x_0+\hat{\lambda}h}\Big)\times\\
&&\qquad \times
\big|U(t_0+h,x)-U^0-\sum_{j=0}^n\sigma_{\nu}^{(j)}(\varepsilon(t_0+h),x-\lambda_j^0(t_0+h))\rr_j^0\big|\mbox{d}x\qquad\\
&&\quad +\int_{\min\{\dot{S}_k,\dot{S}_{B(t,x)}\}h}^{\max\{\dot{S}_k,\dot{S}_{B(t,x)}\}h}
|U(t_0+h,x)-U^0-\sum_{j=0}^n\sigma_{\nu}^{(j)}(\varepsilon(t_0+h),x-\lambda_j^0(t_0+h))\rr_j^0|\mbox{d}x\\
&&=O(1)\sigma\big(\sigma+\max_j|\sigma_{\nu-}^{(j)}|\big)\hat{\lambda}h \varepsilon^2
+O(1)\sigma \max_j|\sigma_{\nu-}^{(j)}|\,  h \varepsilon^2\,\\
&&=O(1)\sigma\big(\sigma+\max_j|\sigma_{\nu-}^{(j)}|\big)h\varepsilon^2.
\end{eqnarray*}
\begin{figure}[tbh]
\begin{centering}
\includegraphics[width=5.9in, height=3.1in]{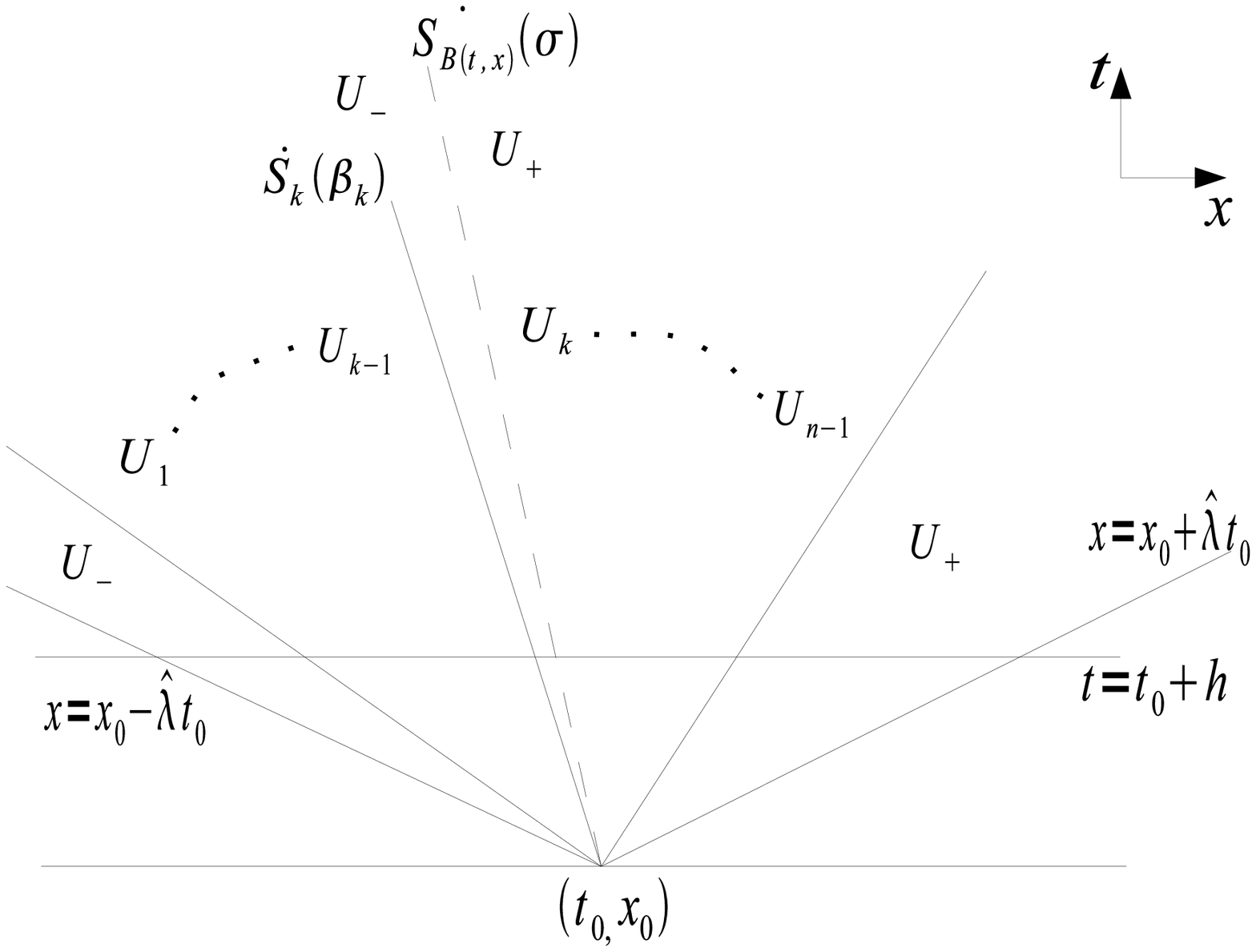}
\caption{Comparison for the genuinely nonlinear case: shock wave}
\label{fig:nonlinear shock}
\end{centering}
\end{figure}

\begin{figure}[tbh]
\begin{centering}
\includegraphics[width=6.2in, height=3.2in]{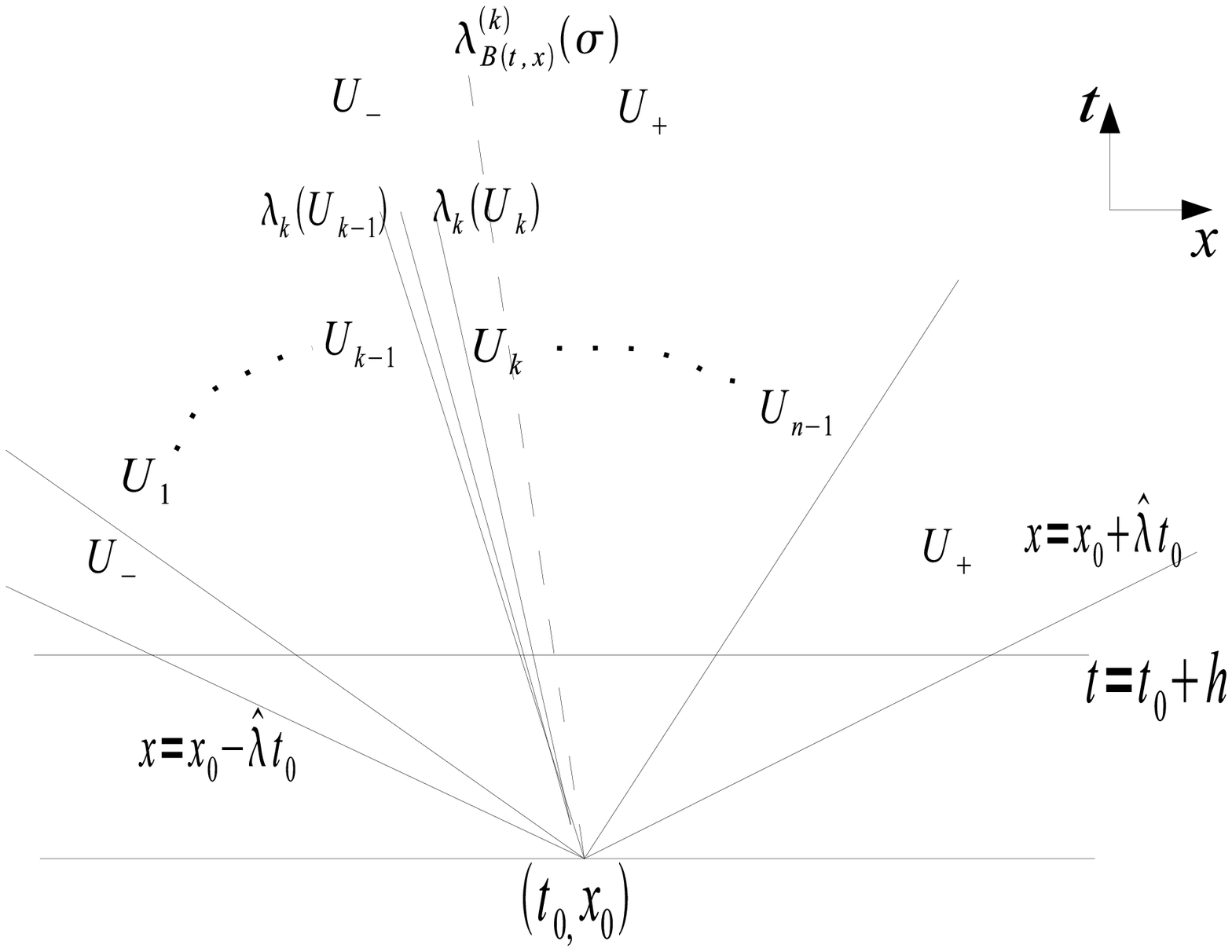}
\caption{Comparison for the genuinely nonlinear case: rarefaction wave}
\label{fig:nonlinear rarefaction}
\end{centering}
\end{figure}

We now consider the case $\sigma>0$.
As for the case $\sigma<0$, we obtain that $\beta_k>0$ when $\varepsilon$ sufficiently small.
Then the $k$-th wave is a rarefaction wave.
From Lemmas \ref{prop:Riemann systems} and \ref{lem:genius nonlinear riemann problem},
we have
\begin{eqnarray}
\lambda_k(U_{k-1})&=&\lambda_{k}(U_-)+O(1)
\sigma\big(\sigma+\max_j|\sigma_{\nu-}^{(j)}|\big)\varepsilon^2,\nonumber\\[2mm]
\lambda_k(U_{k})&=&\lambda_{k}(U_+)+O(1)
\sigma\big(\sigma+\max_j|\sigma_{\nu-}^{(j)}|\big)\varepsilon^2,\nonumber\\[2mm]
\lambda_k(U_{k})&=&\lambda_k(U_{k-1})
+\beta_k
+O(1)\sigma\big(\sigma+\max_j|\sigma_{\nu-}^{(j)}|\big)\varepsilon^2\nonumber\\[2mm]
&=&\lambda_k(U_{-})+\frac{1}{2}\sigma\varepsilon
+O(1)\sigma\big(\sigma+\max_j|\sigma_{\nu-}^{(j)}|\big)\varepsilon^2. \nonumber
\end{eqnarray}
These yield
$$
\lambda_k(U_k)-\lambda_k(U_{k-1})=\frac{1}{2}\sigma\varepsilon
+O(1)\sigma\big(\sigma+\max_j|\sigma_{\nu-}^{(j)}|)\big)\varepsilon^2=O(1)\sigma \varepsilon,
$$
$$
\begin{array}{l}
\lambda_k(U_k)-\lambda^{(k)}_{B(t,x)}(\sigma)\\[2mm]
=\lambda_k(U_-)+\frac{1}{2}\sigma\varepsilon
-\lambda^0_k-\sigma_{\nu-}^{(k)}\varepsilon-\frac{\sigma}{2}\varepsilon
+O(1)\sigma\big(\sigma+\max_j|\sigma_{\nu-}^{(j)}|\big)\varepsilon^2\\[2mm]
=\lambda_k(U^0+\varepsilon\sum_{j=1}^n\sigma_{\nu-}^{(j)}r_j^0)
-\lambda_k^0-\sigma_{\nu-}^{(k)}\varepsilon
+O(1)\sigma\big(\sigma+\max_j|\sigma_{\nu-}^{(j)}|\big)\varepsilon^2\\[2mm]
=O(1)\varepsilon\sum_{j\ne k}(\rr_j\cdot\nabla)\lambda_k^0\sigma_{\nu-}^{(j)}
+O(1)\sigma\varepsilon+O(1)\sigma\big(\sigma+\max_j|\sigma_{\nu-}^{(j)}|\big)\varepsilon^2\\[2mm]
=O(1)\big(\sigma+\max_j|\sigma_{\nu-}^{(j)}|\big)\varepsilon,
\end{array}
$$
and
$$
\begin{array}{l}
\lambda_k(U_{k-1})-\lambda^{(k)}_{B(t,x)}(\sigma)\\[2mm]
=\lambda_k(U_-)-\lambda^0_k-\sigma_{\nu-}^{(k)}\varepsilon
-\frac{\sigma}{2}\varepsilon+O(1)\sigma\big(\sigma+\max_j|\sigma_{\nu-}^{(j)}|\big)\varepsilon^2\\[2mm]
=\lambda_k(U^0+\varepsilon\sum_{j=1}^n\sigma_{\nu-}^{(j)}\rr_j^0)-\lambda_k^0
-\sigma_{\nu-}^{(k)}\varepsilon
-\frac{\sigma}{2}\varepsilon +O(1)\sigma\big(\sigma+\max_j|\sigma_{\nu-}^{(j)}|\big)\varepsilon^2\\[2mm]
=O(1)\big(\sigma+\max_j|\sigma_{\nu-}^{(j)}|\big)\varepsilon,
\end{array}
$$
where we have used the above estimates
and the properties of the centered rarefaction waves:
$$
U(t,x)=
\left\{
\begin{array}{l l}
U_- & \quad \mbox{if} \, \frac{x}{t}<\lambda_k(U_-),\\
U_+ & \quad \mbox{if} \, \frac{x}{t}>\lambda_k(U_-),\\
R_k(\sigma)(U_-) & \quad \mbox{if} \, \frac{x}{t}=\lambda_k(R_k(\sigma)(U_-))\in[\lambda_k(U_-),\lambda_k(U_+)],
\end{array}
\right.
$$
when $U_{\pm}$ are connected by the $k$-rarefaction curve,
and $R_k(\sigma)(U_-)$ is a smooth function for $\sigma$ and $U_-$.

As shown in Figure \ref{fig:nonlinear rarefaction}, similar to the case $\sigma<0$, we have
\begin{eqnarray*}
&&\int_{x_0-\hat{\lambda}h<x<x_0+\hat{\lambda}h}
|S_h(U_{w,\nu}^{\varepsilon}(t_0,\cdot))-U_{w,\nu}^{\varepsilon}(t_0+h,x)|\mbox{d}x\\
&&=\int_{x_0-\hat{\lambda}h<x<x_0+\hat{\lambda}h}
\big|U(t_0+h,x)-U^0-\sum_{j=0}^n\sigma_{\nu}^{(j)}(\varepsilon(t_0+h),x-\lambda_j^0(t_0+h))\rr_j^0\big|\mbox{d}x\\
&&=\int_{x_0-\hat{\lambda}h}^{{\min\{\lambda_k(U_{k-1}),\lambda_{B(t,x)}^{(k)}\}}}
\big|U(t_0+h,x)-U^0-\sum_{j=0}^n\sigma_{\nu}^{(j)}(\varepsilon(t_0+h),x-\lambda_j^0(t_0+h))\rr_j^0\big|\mbox{d}x\\
&&\quad +\int_{{\max\{\lambda_k(U_{k}),\lambda_{B(t,x)}^{(k)}\}}}^{x_0+\hat{\lambda}h}
\big|U(t_0+h,x)-U^0-\sum_{j=0}^n\sigma_{\nu}^{(j)}(\varepsilon(t_0+h),x-\lambda_j^0(t_0+h))\rr_j^0\big|\mbox{d}x\\
&&\quad +\int_{\min\{\lambda_k(U_{k-1}),\lambda_{B(t,x)}^{(k)}\}}^{\max\{\lambda_k(U_{k}),\lambda_{B(t,x)}^{(k)}\}}
|U(t_0+h,x)-U^0-\sum_{j=0}^n\sigma_{\nu}^{(j)}(\varepsilon(t_0+h),x-\lambda_j^0(t_0+h))\rr_j^0|\mbox{d}x\\
&& =O(1)\sigma\big(\sigma+\max_j|\sigma_{\nu-}^{(j)}|\big)\hat{\lambda}h \varepsilon^2
+O(1)\sigma \max_j|\sigma_{\nu-}^{(j)}|h\varepsilon^2\\
&& =O(1)\sigma(\sigma+\max_j|\sigma_{\nu-}^{(j)}|)h\varepsilon^2.
\end{eqnarray*}
This completes the proof.
\end{proof}

\subsection{Comparison for the Linearly Degenerate or Constant
Multiplicity Case}\label{subsec:comparisonlinearlydegenerate}

In this case, we first show that the characteristic fields whose eigenvalue has constant multiplicity, bigger than one,
must be linearly degenerate.

\begin{lem}\label{lem:3.3a}
Assume that $F(U)$ satisfies the assumptions in Theorem {\rm \ref{thm:compacted initial data}}, and
$\lambda_{m+1}(U)\equiv\cdots\equiv\lambda_{m+p}(U)$ for $p>1$.
Then the characteristic fields of $\lambda_{m+1}(U), \cdots, \lambda_{m+p}(U)$ must be linearly degenerate.
\end{lem}

\begin{proof}
Assume that $\lambda_{m+1}(U)\equiv\cdots\equiv\lambda_{m+p}(U)=\lambda(U)$,
and the corresponding eigenvectors $\rr_{m+1}(U),\cdots, \rr_{m+p}(U)$ are linearly independent.
Performing $\rr_j\cdot\nabla$ to both sides of the equation $\nabla F(U)\, \rr_k(U)=\lambda_k(U) \rr_k(U)$, we have
\begin{equation}\label{3.6a}
\nabla^2F(\rr_k(U),\rr_j(U))+\nabla F(\rr_j(U)\cdot\nabla)\rr_k(U)
=(\rr_j(U)\cdot\nabla)\lambda_k(U)\, \rr_k(U)+\lambda_k(U)(\rr_j(U)\cdot\nabla)\rr_k(U),
\end{equation}
where $j,k\in\{m+1,\cdots, m+p\}$ and $j\ne k$.

Taking the dot product on both sides of \eqref{3.6a}
with $\ll_j$ from the left and using the fact that $\ll_j\cdot \rr_k=\delta_{jk}$, we have
$$
\ll_j(U)\cdot\nabla^2F(\rr_k(U),\rr_j(U))+\lambda_j(U) \, \ll_j(U)\cdot(\rr_j(U)\cdot\nabla)\rr_k(U)
=\lambda_k(U)\,  \ll_j(U)\cdot(\rr_j(U)\cdot\nabla)\rr_k(U),
$$
which implies
$$
\ll_j(U)\cdot\nabla^2F(\rr_k(U),\rr_j(U))\equiv 0 \qquad \mbox{if}\, j\ne k.
$$
Next, taking the dot product on both sides of \eqref{3.6a} with $\ll_k$ from the left and
using the fact that $\ll_j\cdot \rr_k=\delta_{jk}$, we have
\begin{eqnarray*}
&&\ll_k(U)\cdot\nabla^2F(\rr_k(U),\rr_j(U))+\lambda_k(U) \ll_k(U)\cdot(\rr_j(U)\cdot\nabla)\rr_k(U)\\[2mm]
&&=\lambda_k(U) \ll_k(U)\cdot(\rr_j(U)\cdot\nabla)\rr_k(U)+(\rr_k(U)\cdot\nabla)\lambda_k(U),
\end{eqnarray*}
which implies
$$
\rr_k(U)\cdot\nabla\lambda_k(U) =\ll_k(U)\cdot\nabla^2F(\rr_k(U),\rr_j(U))\equiv 0.
$$
This completes the proof.
\end{proof}

Now we consider the $k$th-characteristic field whose eigenvalue has constant multiplicity
for $m+1\le k\le m+p$
in \eqref{1.2}.
In this case, the corresponding scalar case is quite simple, and the equations are reduced to be
$\partial_\tau \sigma^{(k)}_{\nu}\equiv 0$.
Thus, $\sigma^{(k)}_{\nu}(\tau,y)$ is a piecewise function and
$$
\sigma^{(k)}_{\nu}(\tau,y)=
\left\{
      \begin{array}{l}
            \sigma^{(k)}_{\nu+} \qquad \mbox{if}~ y>y_0,\\[2mm]
            \sigma^{(k)}_{\nu-} \qquad \mbox{if}~y\le y_0.
      \end{array}
\right.
$$
Then the slope of the discontinuity line in the $(t, x)$--coordinates is
\begin{equation}\label{slope of scalar contact disconuity}
    \lambda^{(k)}_{B(t,x)}(\sigma)
    =\frac{x-x_0}{t-t_0}
    =\frac{y+(\lambda_k^0\tau)/\varepsilon-y_0-(\lambda_k^0\tau_0)/\varepsilon}{\tau/\varepsilon-\tau_0/\varepsilon}
    =\varepsilon\frac{y-y_0}{\tau-\tau_0}+\lambda_k^0
    =\lambda_k^0.
\end{equation}

Next, we consider the Riemann problem to system \eqref{ngo1}  with the corresponding initial
data \eqref{systems corresponding initial data} as before. We obtain

\begin{lem}\label{lem:linearly degerate}
Assume that $F(U)$ satisfies all the assumptions in Theorem {\rm \ref{thm:compacted initial data}}, $m+1\le k\le m+p$ as in \eqref{1.2},
$\sigma_{\nu}^{(k)}$ has a jump at $(\tau_0,y_0)$,
and $\sigma_{\nu+}^{(k)}=\sigma_{\nu-}^{(k)}+\sigma$.
Let $(t_0,x_0)$ be the corresponding point to $(\tau_0,y_0)$,
$U^{\varepsilon}_{\pm}=U^0+\varepsilon\sum_{j=1}^n\sigma_{\nu_{\pm}}^{(j)}\rr_j^0$,
and $U^{\varepsilon}_+=\Phi(U_-^{\varepsilon};\beta_1,\cdots,\beta_n)$.
Then we have
$$
\beta_j(\sigma,\varepsilon)=\delta_{jk}\sigma \varepsilon +O(1)\sigma\big(\max_j|\sigma_{\nu-}^{(j)}|+\sigma\big)\varepsilon^2.
$$
\end{lem}

\begin{proof}
Let $\theta=\varepsilon\sigma$. From the definition, $U^{\varepsilon}_+=U^{\varepsilon}_-+\theta \rr_j^0$, that is,
\begin{equation}\label{3.7a}
\Phi(U_-^{\varepsilon};\beta_1,\cdots,\beta_n)-U_-=\theta \rr_j^0.
\end{equation}
If $\theta=0$, we deduce from $\Phi(U_-^{\varepsilon};\beta_1,\cdots,\beta_n)=U_-$ that $\beta_j|_{\theta=0}=0$.
We differentiate both sides of \eqref{3.7a} with respect to $\theta$ to obtain
$\sum_j\frac{\partial\Phi}{\partial\beta_j}\frac{\partial\beta_j}{\partial\theta}=\rr^0_k$.
From Lemma \ref{prop:Riemann systems}, we know
$\left.\frac{\partial\Phi}{\partial\beta_j}\right|_{\theta=0}
=\left.\frac{\partial\Phi}{\partial\beta_j}\right|_{\beta=0}=\rr_j(U_-^{\varepsilon})$.
Thus, we have
\begin{eqnarray}
&&\sum_j\frac{\partial\Phi}{\partial\beta_j}\Big|_{\theta=0}
 \frac{\partial\beta_j}{\partial\theta}\Big|_{\theta=0}\nonumber \\
&&=\rr_k^0 \nonumber\\
&&=\rr_k(U^{\varepsilon}_- -\varepsilon\sum_{j=1}^n\sigma_{\nu-}^{(j)}\rr_j^0)\nonumber \\
&&=\rr_k(U_-^{\varepsilon})
-\varepsilon\sum_{j=1}^n(\rr_j^0\cdot\nabla)\rr_k(U_-^{\varepsilon})\sigma_{\nu-}^{(j)}+O(1)\varepsilon^2 \nonumber
\\
&&=\rr_k(U_-^{\varepsilon})+O(1)\max_j|\sigma_{\nu-}^{(j)}|\,\varepsilon. \label{3.8a}
\end{eqnarray}
Then we take the dot product on both sides of \eqref{3.8a} with $\ll_j(U_-^{\varepsilon})$ to obtain
$$
\left.\frac{\partial\beta_j}{\partial\theta}\right|_{\theta=0}
=\delta_{jk}+O(1)\max_j|\sigma_{\nu-}^{(j)}|\varepsilon,
$$
and then
\begin{eqnarray*}
\beta_j(\theta,\;\varepsilon)=\delta_{jk}\theta
+O(1)\max_j|\sigma_{\nu-}^{(j)}|\theta\varepsilon+O(1)\theta^2\\
=\delta_{jk}\sigma\varepsilon +O(1)\sigma\big(\max_j|\sigma_{\nu-}^{(j)}|+\sigma\big)\varepsilon^2.
\end{eqnarray*}
This completes the proof.
\end{proof}

With Lemmas \ref{lem:3.3a}--\ref{lem:linearly degerate},
we have

\begin{prop}\label{lem:local L1 estimates for linearly degenerate}
Assume that $F(U)$ satisfies the assumptions in Theorem {\rm \ref{thm:compacted initial data}},
either the $k$-th characteristic field is
linearly degenerate or its eigenvalue has constant multiplicity,
$\sigma_{\nu}^{(k)}$ is a piecewise constant function which satisfies
$\sigma_{\nu+}^{(k)}=\sigma_{\nu-}^{(k)}+\sigma$
and has a jump point at $(\tau_0,y_0)$, $(t_0,x_0)$ is the corresponding point,
and $(\tau_0,y_0)$ is not a jump point of $\sigma_{\nu}^{(j)}$ for $j\ne k$.
Then, for every $\hat{\lambda}>2\max_{|u|\le M}\lambda_k(u)$ with
$M$ being the maximum of the solution to \eqref{ngo1},
if $h$ is sufficiently small, we have
$$
\int_{x_0-\hat{\lambda}h<x<x_0+\hat{\lambda}h}
\big|S_h(U_{w,\nu}^{\varepsilon}(t_0,\cdot))-U_{w,\nu}^{\varepsilon}(t_0+h,\cdot)\big|
\mbox{d}x
\le C\sigma\big(\sigma+\max_j|\sigma_{\nu-}^{(j)}|\big)h\varepsilon^2,
$$
where $U_{w,\nu}^{\varepsilon}$ is defined in Lemma {\rm \ref{lem:local L1 estimates for genius nonlinear}},
and $C>0$ is a constant independent of $h$ and $\varepsilon$.
\end{prop}

\begin{proof}
The proof of this lemma is similar to the one for Lemma \ref{lem:local L1 estimates for genius nonlinear},
especially for the shock case.
The only change is about the shock speed $\dot{S}(\beta_k)$. This can be calculated directly as follows:
$$
\dot{S}(\beta_k)=\lambda_k(U_{k-1})
=\lambda_k(U_-)+O(1)\sigma\big(\sigma+\max_j|\sigma_{\nu_-}^{(j)}|\big)\varepsilon^2
=\lambda_k^0+O(1)\max_j|\sigma_{\nu-}^{(j)}|\,\varepsilon,
$$
and then
\begin{eqnarray*}
\dot{S}(\beta_k)-\dot{S}_{B(t,x)}(\sigma)
&=&\lambda_k^0+O(1)\max_j|\sigma_{\nu-}^{(j)}|\,\varepsilon-\lambda_k^0\\
&=&O(1)\max_j|\sigma_{\nu-}^{(j)}|\,\varepsilon.
\end{eqnarray*}
With this, we can obtain the desired estimate.
\end{proof}

\section{Proof of Theorem \ref{thm:compacted initial data}, Part I:  $0\leq t\leq T_0$}\label{sec:main thm}

We now prove Theorem \ref{thm:compacted initial data} when $t\in [0, T_0]$, where
$T_0$ is defined as follows:
Since $\mbox{supp}(\sigma^{(j)}\big|_{\tau=0})$ is
a compact set for each $j$ (due to the compactness of the support of $U^1$)
and the speed $\lambda_k^0\neq\lambda_j^0$ for $k\neq j$,
there exists $T_0>0$, independent of $\varepsilon\in(0,\varepsilon_0]$, such that,
for $t\geq T_0$, the compact sets
$$
K_j(t)=\{x\in \mathrm{R}\, :\, \sigma^{j}(\varepsilon t, x-\lambda_j^0t)\neq0\}, \quad j=1, 2, \cdots, n,
$$
are disjoint, that is, for $t\geq T_0$,
$$
K_j(t)\cap K_k(t)=\phi\qquad\mbox{if  $j\neq k$.}
$$
In order to establish Theorem \ref{thm:compacted initial data} for $0\leq t\leq T_0$,
a technical lemma is stated here to restrict the integral interval for the
semigroup estimates to be finite.

\begin{lem}\label{lem:local semigroup estimate}
Let $S:\;\mathcal{D}\times[0,\infty)\mapsto\mathcal{D}$ be the semigroup generated by
system \eqref{ngo1}, and let $\hat{\lambda}$ be an upper bound for all the wave speeds,
where $F(U)$ satisfies all the assumptions in Theorem {\rm \ref{thm:compacted initial data}}.
Given any interval $I_0:= [a,b]$, define
$$
I_t:= [a+\hat{\lambda}t,b-\hat{\lambda}t],\quad t<\frac{b-a}{2\hat{\lambda}}.
$$
Then, for every Lipschitz continuous map $W:\;[0,T]\mapsto\mathcal{D}$,
\begin{equation}
\|W(t)-S_tW(0)\|_{L^1(I_t)}
\le L\,\int_0^t\left\{\liminf_{h\rightarrow 0+}\frac{\|W(\tau+h)-S_hW(\tau)\|_{L^1(I_{\tau+h})}}{h}\right\}
\mbox{d}\tau.
\end{equation}
\end{lem}

\begin{proof}
First, for every $\bar{U},\bar{V}\in\mathcal{D}$ and $t\ge 0$, we have
$$
\int_{a+\hat{\lambda}t}^{b-\hat{\lambda}t}|S_t\bar{U}(x)-S_t\bar{V}(x)|\mbox{d}x
\le L\int_a^b|\bar{U}(x)-\bar{V}(x)|\mbox{d}x.
$$
This can be obtained by using the following facts:

(i) If two initial conditions $\bar{U},\bar{V}\in\mathcal{D}$
coincide on $(a,b)$, then $S_t\bar{U}(x)=S_t\bar{V}(x)$ for all $x\in (a+\hat{\lambda}t,b-\hat{\lambda}t)$;

(ii) If $\bar{U}\in\mathcal{D}$, then $\bar{U}\,\chi_{[a,b]}\in\mathcal{D}$.

\medskip
With this estimate, the other steps for the proof is the same as the ones for \eqref{semigroup estimate}
in \cite{b-book2000} (Theorem 2.9).
\end{proof}

Now we prove Theorem \ref{thm:compacted initial data} for $0\leq t\leq T_0$.

\begin{proof}[$\mathbf{Proof~of~Theorem~\ref{thm:compacted initial data}~for~ 0\leq t\leq T_0}$]
We divide the proof into three steps. The constant $C>0$ below is a universal bound independent of
the parameters $(j, \nu, \varepsilon, \tau, t, J)$ in the proof.

\medskip
{\it $Step~1. ~Approximation~by~piecewise~constant~functions.$}
Since $U^1\in BV$, we can construct piecewise constant
functions $\sigma^{(j)}_{\nu}(0,y)$, which satisfy
$$
\begin{array}{l}
TV(\sigma^{(j)}_{\nu}(0,\cdot))\le TV(\sigma^{(j)}(0,\cdot)),\\[2mm]
\|\sigma^{(j)}_{\nu}(0,\cdot)-\sigma^{(j)}(0,\cdot)\|_{\infty}
+\|\sigma^{(j)}_{\nu}(0,\cdot)-\sigma^{(j)}(0,\cdot)\|_{L^1}\le\varepsilon,
\end{array}
$$
where $\sigma^{(j)}(0,y)=\ll_j(U^0)\cdot U^1(y)$.
Then
$$
TV(\sigma^{(j)}_{\nu}(0,\cdot))\le TV(\sigma^{(j)}(0,\cdot))\le C\,TV(U^1)
$$
and
$$
\|\sigma^{(j)}_{\nu}(0,\cdot)\|_{L^\infty}\le 2\|\sigma^{(j)}(0,\cdot)\|_{L^\infty}\le C\|U^1(\cdot)\|_{L^\infty}.
$$

In addition, the jump points of each $\sigma^{(j)}_{\nu}$ are finite,
and the range of those functions is contained in the set $2^{-\nu}\mathbb{Z}$, for some fixed integer $\nu$.
For this initial data, we can construct an approximate solution $\sigma_{\nu}^{(j)}(\tau,y)$ by the method introduced in \S \ref{sec:scalar scheme}. Then the solution $\sigma_{\nu}^{(j)}(\tau,y)$ satisfies
\begin{eqnarray}
&&TV(\sigma_{\nu}^{(j)}(\tau,\cdot))\le C\,TV(\sigma_{\nu}^{(j)}(0,\cdot))\le C\,TV(U^1(\cdot)),\label{4.2a}\\[2mm]
&&\|\sigma_{\nu}^{(j)}(\tau,\cdot)\|_{L^{\infty}}\le C\,\|\sigma_{\nu}^{(j)}(0,\cdot)\|_{L^{\infty}}
\le C \|U^1(\cdot)\|_{L^{\infty}},\label{4.2b}\\[2mm]
&&\|\sigma^{(j)}_{\nu}(\tau,\cdot)-\sigma^{(j)}_{\nu}(s,\cdot)\|_{L^1}\le L|\tau-s|. \label{4.2c}
\end{eqnarray}
Then, from Helly's theorem, for every $\tau$,
there exists a subsequence of functions (still denoted as) $\sigma_{\nu}^{(j)}(\tau, y)$ so that
$$
\sigma_{\nu}^{(j)}(\tau,y)\longrightarrow\sigma_{\tau}^{(j)}(y) \qquad \mbox{for~a.e.}~y\in\mathbb{R},
$$
and
$$
\|\sigma_{\tau}^{(j)}(\cdot)\|_{L^{\infty}}\le C \|U^1(\cdot)\|_{L^{\infty}},
 \qquad TV(\sigma_{\tau}^{(j)}(\cdot))\le C\, TV(U^1(\cdot)).
$$
By standard argument, there exists a subsequence $\sigma_{\nu_j}^{(j)}$
such that $\sigma_{\nu_j}^{(j)}(\tau,\cdot)\rightarrow \sigma^{(j)}(\tau,\cdot)$ pointwise, which is also in
$L^1_{loc}(\mathbb{R};\mathbb{R}^n)$, at every rational time $\tau>0$.
Define
$$
\sigma^{(j)}(\tau,y)=\lim_{m\rightarrow\infty}\sigma^{(j)}(\tau_m,y),
$$
where $\tau_m$ is a rational time and $\tau_m\rightarrow \tau$.
Then, by continuity of time in \eqref{4.2c}, we conclude that $\sigma_{\nu_j}^{(j)}\rightarrow\sigma^{(j)}$ in
$L^1_{loc}([0,\infty)\times\mathbb{R}; \mathbb{R}^n)$, and $\sigma^{(j)}$
satisfies
\begin{eqnarray*}
&&\int_{-\infty}^{\infty}|\sigma^{(j)}(\tau,y)-\sigma^{(j)}(s,y)|\mbox{d}y
\le L|\tau-s|\qquad \mbox{for~all}~\tau,s\ge 0,\\
&& TV(\sigma^{(j)}(\tau,\cdot))\le C\,TV(U^1),
\quad \|\sigma^{(j)}(\tau,\cdot)\|_{L^{\infty}}\le C\|U^1(\cdot)\|_{L^{\infty}}\qquad \mbox{for~all}~\tau\ge 0.
\end{eqnarray*}

\medskip
{\it $Step~2$. Estimate of the term
$\|S_hU_{w,\nu}^{\varepsilon}(t,\cdot)-U_{w,\nu}^{\varepsilon}(t+h,\cdot)\|_{L^1(I)}$ for $h$ small enough,
where $I:=[a,b]$,
$$
U_{w,\nu}^{\varepsilon}(t,x)
=U^0+\varepsilon\sum_{j=1}^n\sigma_{\nu}^{(j)}(\varepsilon t,x-\lambda_j^0t)\rr_j^0,
$$
 and $(t,x)$ is not a point of interaction of
$\sigma_{\nu}^{(j)}(\varepsilon t,x-\lambda_j^0t)$ for all $x\in\mathbb{R}$ and $j=1,\cdots,n$.}

First, consider the case that $(t,x_0)$ is a jump point of one and only one function
of $\big\{\sigma_{\nu}^{(j)}(\tau,y)\big\}_{j=1}^n$, say
$\sigma_{\nu}^{(j)}(\tau,y)$.
From \S \ref{sec:riemann comparison}, we can derive the following estimate:
\begin{equation}\label{1}
\|S_h(U_{w,\nu}^{\varepsilon}(t,\cdot))-U_{w,\nu}^{\varepsilon}(t+h,\cdot)\|_{L^1(I)}
\le C\big|\sigma^{(j)}_{\nu+}-\sigma^{(j)}_{\nu-}\big| {
(\max_j|\sigma_{\nu-}^{(j)}|+\max_j|\sigma_{\nu+}^{(j)}|)}h \varepsilon^2,
\end{equation}
where $\sigma^{(j)}_{\nu\pm}:=\lim_{x\rightarrow x_0\pm}\sigma^{(j)}_{\nu}(\varepsilon t,x-\lambda_j^0t)$,
when $h$ is so small that there is no interaction, and $I$ contains only the wave fronts generated at point
$(t,x_0)$.

\smallskip
We now consider the general case.
Without loss of generality, assume that $(t,x_0)$ is a jump point
of $\sigma_{\nu}^{(j)}$ and $\sigma_{\nu}^{(k)}$.

{\it Claim. Assume that $F(U)$ satisfies all the assumptions in Theorem {\rm \ref{thm:compacted initial data}},
$j$ and $k$ are two distinct integers, $(\tau_0,y_0)$ is a jump point of and only of
$\sigma_{\nu}^{(j)}$ and $\sigma_{\nu}^{(k)}$. Let
$$
\sigma_{\nu\pm}^{(j)}=\lim_{y\rightarrow y_0\pm}\sigma_{\nu}^{(j)}(\tau_0,y)
$$
and
$$
\sigma_{\nu\pm}^{(k)}=\lim_{y\rightarrow y_0\pm}\sigma_{\nu}^{(k)}(\tau_0,y),
$$
$(t_0,x_0)$ is the corresponding point to $(\tau_0,y_0)$,
$$
U^{\varepsilon}_{\pm}=U^0+\varepsilon\sum_{j=1}^n\sigma_{\nu_{\pm}}^{(j)}\rr_j^0, \quad
U^{\varepsilon}_+=\Phi(U_-^{\varepsilon};\beta_1,\cdots,\beta_n).
$$
Then we have
\begin{eqnarray*}
\beta_m&=&\big(\delta_{mj}|\sigma_{\nu+}^{(j)}-\sigma_{\nu-}^{(j)}|+
\delta_{mk}|\sigma_{\nu+}^{(k)}-\sigma_{\nu-}^{(k)}|\big)\varepsilon\\[2mm]
&& +O(1)(|\sigma_{\nu+}^{(j)}-\sigma_{\nu-}^{(j)}|+|\sigma_{\nu+}^{(k)}-\sigma_{\nu-}^{(k)}|) {(\max_j|\sigma_{\nu-}^{(j)}|+\max_j|\sigma_{\nu+}^{(j)}|)}\,\varepsilon^2.
\end{eqnarray*}
}

This property can be proved by combining the properties in the proof of
Lemma \ref{lem:genius nonlinear riemann problem} and the ideas for the proof
of Lemma \ref{lem:linearly degerate}. We omit the details.

With this, we can show
\begin{eqnarray}\label{2}
&&\|S_h(U_{w,\nu}^{\varepsilon}(t,\cdot))-U_{w,\nu}^{\varepsilon}(t+h,\cdot)\|_{L^1(I)}\notag\\
&&\,\, \le C \big(|\sigma^{(j)}_{\nu+}-\sigma^{(j)}_{\nu-}|+|\sigma^{(k)}_{\nu+}-\sigma^{(k)}_{\nu-}|\big) {(\max_j|\sigma_{\nu-}^{(j)}|+\max_j|\sigma_{\nu+}^{(j)}|)} h\varepsilon^2,
\end{eqnarray}
when $h$ is so small that there is no interaction and $I$ contains only the wave fronts generated
at the point $(t,x_0)$.
The proof is similar to the one in Propositions \ref{lem:local L1 estimates for genius nonlinear}
and \ref{lem:local L1 estimates for linearly degenerate};
the difference here is that we should divide the interval into three parts:
the additional part is near the $j$th-wave.
Thus, we estimate not only $\dot{S}(\beta_k)-\dot{S}_{B(t,x)}(\sigma)$
or $\lambda_k(U_k)-\lambda^{(k)}_{B(t,x)}(\sigma)$ as in
Proposition \ref{lem:local L1 estimates for genius nonlinear} or
\ref{lem:local L1 estimates for linearly degenerate},
but also $\dot{S}(\beta_j)-\dot{S}_{B(t,x)}(\sigma)$
or $\lambda_j(U_j)-\lambda^{(j)}_{B(t,x)}(\sigma)$.

With \eqref{1} and \eqref{2},
we can derive an estimate for arbitrary finite interval $I$ by dividing the interval into
some subintervals, every one of which contains only one jump point.
Since the number of these subsets is finite,  we can sum together to obtain
\begin{equation}\label{3}
\|S_h(U_{w,\nu}^{\varepsilon}(t,\cdot))-U_{w,\nu}^{\varepsilon}(t+h,\cdot)\|_{L^1(I)}\le
C\sum_{j=1}^n TV_{(t,x)\in\{t\}\times I}(\sigma^{(j)}_{\nu}(\varepsilon t,y)){\,\max_j|\sigma_{\nu}^{(j)}|h\varepsilon^2},
\end{equation}
where $(\tau,y)=(\varepsilon t, x-\lambda_j^0t)$.

\medskip
{\it $Step~3$. Completion of the proof}. For an arbitrary finite interval $J=[a,b]$,
let $I=[a-\hat{\lambda}t,b+\hat{\lambda}t]$.
Then, from the definition in Lemma \ref{lem:local semigroup estimate}, we have $I_t=J$.
For any initial data $U^1(x)\in BV\cap L^1$,
$\sigma^{(j)}$ and $U^{\varepsilon}_{w}(t,x)$ are defined from \eqref{scaler equation},
\eqref{og expansion}, and \eqref{scaler initial data}.
Then, following Step 1, we construct piecewise constant functions
$\sigma_{\nu}^{(j)}(\tau,y)$ so that
\begin{eqnarray*}
\|U^{\varepsilon}(t,\cdot)-U^{\varepsilon}_w(t,\cdot)\|_{L^1(J)}
&\le& \|U^{\varepsilon}(t,\cdot)-U^{\varepsilon}_{\nu}(t,\cdot)\|_{L^1(J)}
    +\|U^{\varepsilon}_w(t,\cdot)-U^{\varepsilon}_{w,\nu}(t,\cdot)\|_{L^1(J)}\\[2mm]
&& \;  +\|U^{\varepsilon}_{\nu}(t,\cdot)-U^{\varepsilon}_{w,\nu}(t,\cdot)\|_{L^1(J)}\\[2mm]
&=& I_1+I_2+I_3.
\end{eqnarray*}
For $I_1$, from Lemma \ref{prop:stability},
    \begin{eqnarray}
    I_1&=&\|S_t(U^0+\varepsilon U^1(\cdot))-S_t(U^0+\varepsilon\sum_{j=1}^n\sigma_{\nu}^{(j)}(0,\cdot)\rr_j^0)\|_{L^1(J)}\\[2mm]
    &\le &L\|U^0+\varepsilon U^1(\cdot)-\big(U^0+\varepsilon\sum_{j=1}^n\sigma_{\nu}^{(j)}(0,\cdot)\rr_j^0\big)\|_{L^1}\\[2mm]
    &\le &C\sum_{j=1}^n\|\sigma^{(j)}(0,\cdot)-\sigma^{(j)}_{\nu}(0,\cdot)\|_{L^1}.
    \end{eqnarray}
Hence, $I_1\rightarrow 0$ when $\nu\rightarrow \infty$.

For $I_2$, it is reduced to estimate $\|\sigma_{\nu}^{(j)}(\tau,\cdot)-\sigma^{(j)}(\tau,\cdot)\|_{L^1(J)}$.
   From Step 1, we can select a rational time sequence $\{\tau_m\}_{m=1}^{\infty}$
   and a subsequence (still denoted)  $\{\nu\}$ such that
   $\tau_m\rightarrow\tau$ when $m\rightarrow\infty$
   and, for all $m$, $\sigma_{\nu}^{(j)}(\tau_m,y)\rightarrow\sigma^{(j)}(\tau_m,y)$ for a.e. $y\in\mathbb{R}$,
   when $\nu\rightarrow\infty$.
   Then
   $$
   \|\sigma_{\nu}^{(j)}(\tau_m,\cdot)-\sigma^{(j)}(\tau_m,\cdot)\|_{L^1(J)}\rightarrow 0 \qquad\mbox{when}\quad
   \nu\rightarrow\infty.
   $$
   Using the Lipschitz continuity \eqref{lip continuity}, we obtain that
    $I_2\rightarrow 0$ when $\nu\rightarrow\infty$.

    For $I_3$, from Lemma \ref{lem:local semigroup estimate}, we obtain
    \begin{eqnarray*}
    I_3&=&\|U_{\nu}^{\varepsilon}(t,\cdot)-U^{\varepsilon}_{w,\nu}(t,\cdot)\|_{L^1(I_t)}\\
    &\le& L\int_0^t\liminf_{h\rightarrow 0+}
    \frac{\|U_{w,\nu}^{\varepsilon}(s+h,\cdot)-S_h(U_{w,\nu}^{\varepsilon}(s,\cdot))\|_{L^1(I_{s+h})}}{h}
    \mbox{d}s\\
    &\le& L\int_0^t\liminf_{h\rightarrow 0+}\frac{C\varepsilon^2h\sum_{j=1}^{n}TV(\sigma_{\nu}^{(j)}(\varepsilon s,\cdot)) \,{\max_j||\sigma_{\nu}^{(j)}(\varepsilon s,\cdot)||_{\infty}}}{h}\mbox{d}s\\
    &\le& C \varepsilon^2 \int_0^t  {\big(TV(U^1)\big)^2} \mbox{d}s\\
    &=& C\, \big(TV(U^1)\big)^2t \varepsilon^2,
    \end{eqnarray*}
    where we have used the fact that $TV(\sigma_{\nu}^{(j)}(\varepsilon t,\,\cdot))\leq C\,TV(\sigma^{(j)}_{\nu}(0,\,\cdot))$
    for all $t>0$.
    With these three estimates together, passing limit $\nu\rightarrow\infty$, we obtain that,
    for any arbitrary finite interval $J:= [a,b]$,
    $$
    \|U^{\varepsilon}(t,\cdot)-U^{\varepsilon}_w(t,\cdot)\|_{L^1(J)}\le C\, \big(TV(U^1)\big)^2t \varepsilon^2.
    $$
    Finally, let $a\rightarrow-\infty$ and $b\rightarrow\infty$, we have
    {
    \begin{align}
    \|U^{\varepsilon}(t,\cdot)-U^{\varepsilon}_w(t,\cdot)\|_{L^1}
    \le C\, \big(TV(U^1)\big)^2 t \varepsilon^2.
    \end{align}}
Therefore, for $0\leq t\leq T_0$, we obtain
 \begin{align}\label{coreissue}
    \|U^{\varepsilon}(t,\cdot)-U^{\varepsilon}_w(t,\cdot)\|_{L^1}
    \le C\,T_0 \varepsilon^2.
    \end{align}
\end{proof}

As a direct corollary, we conclude
\begin{coll}
Assume that $F$ satisfies the assumptions in Theorem {\rm \ref{thm:compacted initial data}},
and $U^1\in BV\cap L^1(\mathbb{R};\mathbb{R}^n)$ (not necessarily with compact support).
Consider an entropy solution $U^{\varepsilon}(t,x)$ of \eqref{ngo1}, which is the SRS, and
the weakly nonlinear geometric optics expansion $U^{\varepsilon}_w(t,x)$ defined by
\eqref{scaler equation}, \eqref{og expansion}, and \eqref{scaler initial data}.
Then there exists $\varepsilon_0>0$ such that, when $\varepsilon\in (0, \varepsilon_0]$, we have
\begin{equation} \label{4.13}
\|U^{\varepsilon}(t,\cdot)-U^{\varepsilon}_w(t,\cdot)\|_{L^1}\le C t\varepsilon^2
\qquad\,\,\mbox{for some constant $C>0$ independent of $\varepsilon$ and $t$}.
\end{equation}
\end{coll}

However, estimate \eqref{4.13} is not really strong enough to justify
the weakly nonlinear geometric optics approximation as noted in
Schochet \cite{s-jde1994473}.
In \S 6, we further develop the approach to improve \eqref{4.13} into
the stronger error estimate \eqref{ieq:noncompact case} in Theorem 1.2,
which is strong enough indeed to justify the approximation
even for the case when  $U^1\in BV\cap L^1(\mathbb{R};\mathbb{R}^n)$ without compact support.

{
\section{Proof of Theorem \ref{thm:compacted initial data}, Part II: $t\geq T_0$.}\label{sec:compact support}}

Notice from \S  \ref{sec:main thm} that,
when $t\geq T_0$, the compact sets $K_j(t), j=1, 2, \cdots, n$,
are disjoint:
$$
K_j(t)\cap K_k(t)=\phi \qquad\mbox{if  $j\neq k$.}
$$
In this section, we give a careful calculation based on this and
complete the proof of Theorem \ref{thm:compacted initial data}.
The constant $C>0$ below is the universal constant, independent
of $(\varepsilon, \nu, h, t, \tau)$ in the proof.

We first refine expansion \eqref{og expansion}
or \eqref{systems corresponding initial data}
by the following corresponding approximation:
\begin{eqnarray}
V^{\varepsilon}_{\nu}
&=& U^0+\varepsilon\sum_{j}\sigma^{(j)}_{\nu}(\varepsilon t, x-\lambda_j^0t)\rr_j^0 +\frac{\varepsilon^2}{2}\sum_{j\in N}\big(\sigma^{(j)}_{\nu}(\varepsilon t, x-\lambda_j^0t)\big)^2(\rr_j^0\cdot\nabla)\rr_{j}^0\nonumber \\
&& +\varepsilon^2\sum_{j_i\in L}E_{\nu}^{(j)}(x-\lambda_j^0(t-T_0)), \label{auxi function-a}
\end{eqnarray}
where $j\in N$ means that the corresponding $j-$th characteristic field is genuinely nonlinear,
while $j_i\in L$ means that the corresponding characteristic field is linearly degenerate
and all $\{j_i\}$ together constitute the $j$-th characteristic field. Furthermore,
$$
E^{(j)}_{\nu}=\frac{1}{\varepsilon^2}\big(W^{(j)}_{\nu}-U_0-\varepsilon\sigma_{\nu}^{(j)}\rr_j^0\big),
$$
and $W_{\nu}^{(j)}$ is defined as follows.
First set $[a,b]:=\mbox{supp}_{x}\sigma^{(j)}(\varepsilon T_0,x-\lambda_j^{0}T_0)$
and let $x_{j_i}\in[a,b)$ be a point at which $\sigma_{\nu}^{(j)}(\varepsilon T_0,x-\lambda_j^0T_0)$
has a jump. Denote the set of these jump points together by $J_{j}$ and then define the piecewise constant
function $W^{(j)}_{\nu}$ with the same jump points in $J_{j}$.
At each jump point $x_{j_i}\in J_j$,
if
\begin{equation}\label{5.2a}
\varepsilon(\sigma^{(j_i)}_{\nu+}-\sigma^{(j_i)}_{\nu-})\rr_k^0=\Phi(\beta_1,\cdot\cdot\cdot,\beta_n),
\end{equation}
and the multiplicity of the $j$-th characteristic field is $m$,
then the difference of $W^{(j)}_{\nu}$ at this jump point is defined by
\begin{equation}\label{5.3a}
W^{(j)}_{\nu+}-W^{(j)}_{\nu-}:=\Phi(0,\cdot\cdot\cdot,0,\beta_{j_1},\cdot\cdot\cdot\beta_{j_m},0,\cdot\cdot\cdot,0).
\end{equation}
For the point $x<a$ or $x>b$, we define $W^{(j)}_{\nu}=U_0$.
Then we have the property that
$K_j(t)=\mbox{supp}_x\{W^{(j)}_{\nu}(x-\lambda^0_j(t-T_0))-U^0\}$.
From the proof of Lemma \ref{prop:CS} later,
we will see that the error term satisfies that $|E^{(j)}_{\nu}|\leq C$.

\smallskip
Now we consider the case $j\in N$.
As the bootstraps in \S \ref{sec:riemann comparison}, we have the following lemma.

\smallskip
\begin{lem}\label{lem:GN Riemann compact support}
Suppose that the assumptions in Lemma {\rm \ref{lem:genius nonlinear riemann problem}} hold.
Let
$$
V_{\nu\pm}^{\varepsilon}:=U^0+\varepsilon\sigma^{(k)}_{\nu\pm}\rr_{k}^0
+\frac{\varepsilon^2}{2}\big(\sigma_{\nu\pm}^{(k)}\big)^2(\rr_k^0\cdot\nabla) \rr_k^0
$$
with $\Phi(V_{\nu-}^{\varepsilon};\beta_1,\cdot\,\cdot\,\cdot,\beta_n)=V_{\nu+}^{\varepsilon}$. Then we have
\begin{equation}
\beta_j=\sigma\varepsilon\delta_{kj}+O(1)\epsilon^3\big((\sigma_{\nu-}^{(k)})^2\sigma+\sigma_{\nu-}^{(k)}\sigma^2+\sigma^3\big),
\end{equation}
where $\sigma=\sigma_{\nu+}^{(k)}-\sigma_{\nu-}^{(k)}$,
and $\delta_{kj}$ is the Kronecker delta.
\end{lem}

\begin{proof}
Let $\theta=\varepsilon\sigma$. Then
$$
V^{\varepsilon}_{\nu+}-V^{\varepsilon}_{\nu-}=\theta \rr_{k}^0+\frac{\theta^2}{2}(\rr_k^0\cdot\nabla)\rr_k^0
+\varepsilon\theta\sigma_{\nu-}^{(k)}(\rr_k^0\cdot\nabla)\rr_k^0.
$$
Thus we have the following equation:
\begin{equation}\label{equ:star}
\Phi(V^{\varepsilon}_{\nu-}; \, \beta_1,\,\cdot\,\cdot\,\cdot,\,\beta_n)=V^{\varepsilon}_{\nu-}+\theta \rr_k^0+\Big(\frac{\theta^2}{2}+\varepsilon\theta\sigma_{\nu-}^{(k)}\Big)(\rr_k^0\cdot\nabla)\rr_k^0.
\end{equation}
Clearly, by the implicit function theorem, we obtain a unique solution $(\beta_1, \cdots, \beta_n)(\theta)$ of \eqref{equ:star} such that
$\beta_j(0)=0, j=1, \cdots, n$.

\smallskip
Differentiating \eqref{equ:star} with respect to $\theta$ and then letting $\theta=0$, we have
\begin{equation}\label{5.6a}
\sum_{j}\frac{\partial\Phi(V^{\varepsilon}_{\nu-};\,\beta_1,\,\cdot\,\cdot\,\cdot,\,\beta_n)}{\partial\beta_j} \frac{\partial\beta_j}{\partial\theta}\big|_{\theta=0}
=\rr_k^0+\varepsilon\sigma_{\nu-}^{(k)}(\rr_k^0\cdot\nabla)\rr_k^0
  +\theta (\rr_k^0\cdot\nabla)\rr_k^0\big|_{\theta=0}.
\end{equation}
Since $K_i\cap K_j=\phi$ if $i\neq j$, we find that $\sigma^{i}\sigma^{j}=0$ if $i\ne j$.
Using this fact to calculate \eqref{5.6a}, we have
\begin{align*}
\sum_{j}\frac{\partial\Phi}{\partial\beta_j} \frac{\partial\beta_j}{\partial\theta}\big|_{\theta=0} =\sum_j \rr_j(V_{\nu-}^{\varepsilon})\frac{\partial\beta_j}{\partial\theta}\big|_{\theta=0}
=\rr_k(V_-^{\varepsilon})+O(1)(\sigma_{\nu-}^{(k)})^2 \varepsilon^2,
\end{align*}
which yields
\begin{equation}\label{5.3ab}
\frac{\partial\beta_j}{\partial\theta}\big|_{\theta=0}=\delta_{ij}+O(1)(\sigma_{\nu-}^{(k)})^2 \varepsilon^2.
\end{equation}

Next, taking twice derivatives on both sides of \eqref{equ:star} with respect to $\theta$ and then letting $\theta=0$, then
$$
\sum_{i,j}\frac{\partial^2\Phi}{\partial\beta_i\partial\beta_j} \frac{\partial\beta_i}{\partial\theta}\frac{\partial\beta_j}{\partial\theta}\big|_{\theta=0} +\sum_j\frac{\partial\Phi}{\partial\beta_j}\frac{\partial^2\beta}{\partial\theta^2}\big|_{\theta=0}
=(\rr_k^0\cdot\nabla)\rr_k^0.
$$
Thus, we have
$$
(\rr_k(V_{\nu-}^{\varepsilon})\cdot\nabla) \rr_k(V_{\nu-}^{\varepsilon})
+O(1)(\sigma_{\nu-}^{(k)})^2 \varepsilon^2
+\sum_j
\rr_j(V_{\nu-}^{\varepsilon})\frac{\partial^2\beta_j}{\partial\theta^2}\big|_{\theta=0}
=(\rr_k^0\cdot\nabla)\rr_k^0.
$$
Noting that $(\rr_k(V_{\nu-}^{\varepsilon})\cdot\nabla)\rr_k(V_{\nu-}^{\varepsilon})
=(\rr_k^0\cdot\nabla)r_k(U^0)+O(1)\sigma_{\nu-}^{(k)} \varepsilon$,
we obtain
\begin{equation}\label{5.4a}
\frac{\partial^2\beta_j}{\partial\theta^2}\big|_{\theta=0}=O(1)\sigma_{\nu-}^{(k)} \varepsilon.
\end{equation}
Combining \eqref{5.3ab}--\eqref{5.4a} with the Taylor expansion, we finally have
\begin{align*}
\beta_j(\theta)
=&\beta_j(0)+\frac{\partial\beta_j}{\partial\theta}\big|_{\theta=0}\theta +\frac{\partial^2\beta_j}{\partial\theta^2}\big|_{\theta=0}\theta^2+O(1)\theta^3\\
=&\sigma\varepsilon\delta_{kj}+O(1)\big((\sigma_{\nu-}^{k})^2\sigma+\sigma_{\nu-}^{(k)}\sigma^2+\sigma^3\big)\,\varepsilon^3.
\end{align*}
This completes the proof.
\end{proof}

With this better estimate, as in \S \ref{sec:riemann comparison}, we further have

\begin{lem}\label{prop:GN Riemann compact support}
Suppose that the assumptions in Theorem {\rm \ref{thm:compacted initial data}} hold.
Then there exists $h_0>0$ such that, when $h\in (0, h_0]$, for $t \geq T_0$, we have
\begin{eqnarray*}
&& \int_{x_0-\hat{\lambda}h<x<x_0+\hat{\lambda}h} |S_h(V_{\nu}^{\varepsilon}(t_0,x))-V^{\varepsilon}_{\nu}(t_0+h,x)|\mbox{d} x \\
&&\leq C(\max|\sigma_{\nu-}^k|)^2|\sigma|h \varepsilon^3 + C2^{-\nu}|\sigma|h \varepsilon^2.
\end{eqnarray*}
\end{lem}

\begin{proof} We divide the proof into two cases.

\smallskip
{\it Case 1 (shock case):} $\sigma<0$. In this case, by Lemma \ref{lem:GN Riemann compact support},
$\beta_k<0$. Denote $S_k$ the k-shock speed of $\beta_k$ in the $(t,x)$--coordinates and $V_k=\Phi(V_{\nu-}^{\varepsilon};\,\beta_1,\,\cdot\,\cdot\,\cdot,\beta_k,0,\cdot\,\cdot\,\cdot,0)$.
Then
$$
S_k=\lambda_k(V_{k-1})+\frac{1}{2}\beta_k+O(1)|\beta_k|^2
=\lambda_k^0+\sigma_{\nu-}^{(k)}\varepsilon +\frac{1}{2}\sigma\varepsilon +O(1)(\sigma^2+\varepsilon|\sigma|)\,\varepsilon^2.
$$
On the other hand, the shock speed of the Burgers equation of $\sigma^{k}$ in the $(t,x)$--coordinates is
$$
S_{Bk}=\lambda_k^0+\sigma_{\nu-}^{(k)}\varepsilon +\frac{1}{2}\sigma\varepsilon.
$$
Then
$$
S_k-S_{Bk}=O(1)(\sigma^2+|\sigma||\sigma_{\nu-}^{(k)}|)\varepsilon^2.
$$
Meanwhile, with the fact that $|V_{k-1}-V_n|+|V_k-V_0|=O(1)|\sigma|\varepsilon$,
we deduce the core estimate for the shock case:
$$
|S_k-S_{Bk}|\big(|V_{k-1}-V_{\nu+}^{\epsilon}|+|V_k-V_{\nu-}^{\varepsilon}|\big)
=O(1)\big(|\sigma|^3+|\sigma|^2|\sigma_{\nu-}^{(k)}|\big) \varepsilon^3.
$$
Then, with this estimate in hand, following the bootstrap in the proof of
Proposition \ref{lem:local L1 estimates for genius nonlinear} step by step,
using the results in Lemma \ref{lem:GN Riemann compact support}, and performing cumbersome calculations,
we have
\begin{equation}\label{5.5a}
\int_{x_0-\hat{\lambda}h<x<x_0+\hat{\lambda}h}|S_h(V_{\nu}^{\varepsilon}(t,x))-V^{\varepsilon}_{\nu}(t+h,x)| \mbox{d} x
\leq C(\max|\sigma_{\nu-}^k|)^2|\sigma|h \varepsilon^3.
\end{equation}

\medskip
{\it Case 2 (rarefaction case)}: $\sigma>0$. From the scheme for the scalar equation in \S \ref{sec:scalar scheme},
or by Lemma \ref{lem:burger's riemann problem}, we conclude that the strength of discontinuity satisfies
the following property:
$$
\sigma\leq C\, 2^{-\nu}.
$$
First, from Lemma \ref{lem:burger's riemann problem},
the speed of rarefaction front of the Burgers equation
in the $(t,x)$--coordinates is
$$
\lambda_{B}^{(k)}=\lambda_k^0 +\sigma_{\nu-}^{(k)}\varepsilon+\frac{1}{2}\sigma\varepsilon.
$$
On the other hand, the speed of characteristics is
\begin{align*}
&\lambda_k(V_{k-1}) =\lambda_k(V_{\nu-}^{\varepsilon})
   +O(1) \big((\sigma_{\nu-}^{(k)})^2\sigma+\sigma_{\nu-}^{(k)}\sigma^2+\sigma^3\big)\varepsilon^3 ,\\
&\lambda_k(V_{k})=\lambda_k(V_{\nu+}^{\varepsilon})
   +O(1)\big((\sigma_{\nu-}^{(k)})^2\sigma+\sigma_{\nu-}^{(k)}\sigma^2+\sigma^3\big)\varepsilon^3,\\
&\lambda_k(V_{k})-\lambda_k(V_{k-1})=O(1)\sigma\varepsilon
  +O(1)\big((\sigma_{\nu-}^{(k)})^2\sigma+\sigma_{\nu-}^{(k)}\sigma^2+\sigma^3\big)\varepsilon^3.
\end{align*}
Thus, by following exactly the proof of Proposition \ref{lem:local L1 estimates for genius nonlinear},
 direct computation yields that
\begin{equation}\label{5.6ab}
\int_{x_0-\hat{\lambda}h<x<x_0+\hat{\lambda}h}|S_h(V_{\nu}^{\varepsilon}(t_0,x))-V^{\varepsilon}_{\nu}(t_0+h,x)| \mbox{d} x \leq C\sigma^2h \varepsilon^2 \leq C 2^{-\nu}|\sigma|h \varepsilon^2.
\end{equation}

Combining \eqref{5.5a} with \eqref{5.6ab}, we arrive at the result.
\end{proof}

Next, consider the linearly degenerate case, namely, $\{j_{i}\}_{i=1}^m\subset L$.
By construction, all the values of $W^{(j)}_{\nu}$ lie on the curve $S_j(U_0)$,
{\it i.e.} the integral surface of the vector-fields spanned by $\rr_{j_i}$
passing through $U_0$, while $\lambda_j^0$ is the speed. Here $W^{(j)}_{\nu}$ is introduced
below \eqref{auxi function-a} to replace $U^0+\varepsilon\sum_{i=1}^{m}\sigma^{(j_i)}_{\nu}\rr_{j_i}^0$.
Thus, $W^{(j)}_{\nu}$ is exactly a Riemann solution. Then we have

\begin{lem}\label{prop:LD Riemann Compact support}
Suppose that the assumptions in Theorem {\rm \ref{thm:compacted initial data}} hold.
Then there exists $h_0>0$ such that, when $h\in (0, h_0]$, for $t \geq T_0$, we have
$$
\int_{x_0-\hat{\lambda}h<x<x_0+\hat{\lambda}h}|S_h(V_{\nu}^{\varepsilon}(t,x))-V^{\varepsilon}_{\nu}(t+h,x)| \mbox{d} x =0.
$$
\end{lem}
With Lemmas \ref{prop:GN Riemann compact support}
and \ref{prop:LD Riemann Compact support} in hand,
summing up the above estimates and using \eqref{semigroup estimate}, we have

\begin{lem}\label{prop:V}
Suppose that the assumptions in Theorem {\rm \ref{thm:compacted initial data}} hold. Then
$$
\|S(t-T_0)V^{\varepsilon}(T_0)-V^{\varepsilon}(t)\|_{L^1(\mathrm{R})}
\leq C\,\varepsilon^2 \qquad \mbox{when}\,\, t\geq T_0.
$$
\end{lem}

\begin{proof}
We first use the standard error formula \eqref{semigroup estimate} and then let $h$ small enough
such that $S(h)V_{\nu}^{\epsilon}$ is the solution obtained by piecing together the standard entropy
solutions of the Riemann problems determined by the jumps of $V_{\nu}^{\epsilon}$.
Then we can use Lemmas \ref{prop:GN Riemann compact support} and \ref{prop:LD Riemann Compact support}
to obtain
\begin{align*}
&\|S(t-T_0)V_{\nu}^{\varepsilon}(T_0)-V^{\varepsilon}_{\nu}(t)\|_{L^1}\\[2mm]
&\leq L\int_{T_0}^t\liminf_{h\rightarrow0+}
     \frac{\|S(h)V_{\nu}^{\varepsilon}(z)-V_{\nu}^{\varepsilon}(z+h)\|_{L^1}}{h}\mbox{d}z\\
&\leq L\int_{T_0}^t\liminf_{h\rightarrow0+}
     \frac{\sum_{\sigma}|S(h)V_{\nu}^{\varepsilon}(z)-V_{\nu}^{\varepsilon}(z+h)|}{h}\mbox{d}z\\
&\leq O(1)\left(\int_{T_0}^{\max\{T_0,\,\varepsilon^{-1}\}}+\int_{\max\{T_0,\,\varepsilon^{-1}\}}^t\right)
 \Big(\varepsilon^3\sum_{k\in N}\big(TV(\sigma^{(k)}_\nu(\varepsilon z))\big)^3
   +\varepsilon^22^{-\nu}\sum_{k\in N}TV(\sigma^{(k)}_\nu)\Big)\mbox{d}z,
\end{align*}
where $\sum_\sigma$ is the sum of all the jumps at $t=z+h$ with jump strength $\sigma$.
Then, passing the limit $\nu\to\infty$, we have
\begin{align*}
\|S(t-T_0)V^{\varepsilon}(T_0)-V^{\varepsilon}(t)\|_{L^1}
\leq C\,\left(\varepsilon^2+\varepsilon^{3/2}
\int_{\varepsilon^{-1}}^t\frac{\mbox{d}z}{z^{3/2}}\right)\le C \varepsilon^2,
\end{align*}
where we have used the fact that, for genuinely nonlinear scalar conservation laws,
if the initial data $u_0$ has compact support and
satisfies $\|u_0\|_{\infty}\leq M$, then the solution $u(t,x)$ satisfies
$$
TV \big(u(t,\cdot)\big)\leq C\,t^{-\frac{1}{2}},
$$
where the constant $C$ depends only on $u_0$.
This completes the proof.
\end{proof}

Since the estimates for the auxiliary function have been established,
we are now at the stage to consider the estimates for the geometric optic expansion
function $U_{w}^{\varepsilon}$.

\begin{lem}\label{prop:CS}
Suppose that the assumptions in Theorem {\rm \ref{thm:compacted initial data}} hold.
Then
$$
\|S(t-T_0)U^{\varepsilon}_w(T_0)-U^{\varepsilon}_w(t)\|_{L^1(\mathrm{R})}\leq C\,\varepsilon^2
\quad \mbox{when}\,\, t\geq T_0.
$$
\end{lem}

\begin{proof}
This is in fact a simple corollary of Lemma \ref{prop:V}. Notice that
\begin{align*}
U^{\varepsilon}_{w,\nu}(t)=&V^{\varepsilon}_{\nu}(t)
-\frac{1}{2}\varepsilon^2\sum_{j\in N}\big(\sigma^{(j)}(\varepsilon t, x-\lambda_j^0t)\big)^2 (\rr_j^0\cdot\nabla) \rr_j^0\\
&-\varepsilon^2\sum_{j_i\in L}E_{\nu}^{(j)}(x-\lambda_j^0(t-T_0)).
\end{align*}
From Lemma \ref{lem:linearly degerate} and the fact that the discontinuities
of $W^{(j)}_{\nu}$ and $U^0+\varepsilon\sum_{i=1}^m\sigma^{(j_i)}_{\nu}\rr_{j_i}^0$
share the same speed $\lambda_j^0$, we notice that, for any $x\in K_j(T_0)$,
$x-\lambda_j^0(t-T_0)\in K_j(t)$ and
\begin{eqnarray*}
&&\sigma_{\nu}^{(j_i)}(\varepsilon t,x-\lambda_j^0T_0-\lambda_j^0(t-T_0))
=\sigma_{\nu}^{(j_i)}(\varepsilon T_0,x-\lambda_j^0T_0),\\[2mm]
&& W^{(j)}_{\nu}(x-\lambda_j^0(t-T_0))=W^{(j)}_{\nu}(x).
\end{eqnarray*}
Thus, we have
\begin{align*}
&|E^{(j)}_{\nu}(x-\lambda_j^0(t-T_0))|\\
&=\frac{1}{\varepsilon^2}|U^0+\varepsilon\sum_{i=1}^m\sigma_{\nu}^{(j_i)}(\varepsilon T_0,x-\lambda_j^0T_0)\rr_{j_i}^0-W^{(j)}_{\nu}(x)|\\
&\leq \frac{1}{\varepsilon^2}\sum_{\sigma_i\in J_j}|\Phi(0,\cdot\cdot\cdot,0,\beta_{j_1}(\sigma_i),\cdot\cdot\cdot,\beta_{j_m}(\sigma_i),0,\cdot\cdot\cdot,0)
     -\Phi(\beta_1(\sigma_i),\cdot\cdot\cdot,\beta_n(\sigma_i))|\\
&\leq \frac{1}{\varepsilon^2}\sum_{\sigma_i\in J_j,\,k\neq j_i}C\,|\beta_k(\sigma_i)|\\
&\leq\sum_{\sigma_i\in J_j}C\,|\sigma_i|
\leq C,
\end{align*}
where $\sigma_i\in J_j$ means that there is a jump point of $\sigma^{(j_i)}_{\nu}$ at which the strength is $\sigma_i$.

Thus,  by Lemma \ref{prop:stability},
\begin{align*}
    &\|S(t-T_0)U^{\varepsilon}_{w,\nu}(T_0)-U^{\varepsilon}_{w,\nu}(t)\|_{L^1}\\[2mm]
&\leq \|S(t-T_0)U^{\varepsilon}_{w,\nu}(T_0)-S(t-T_0)V_{\nu}^{\varepsilon}(T_0)\|_{L^1}
     +\|S(t-T_0)V_{\nu}^{\varepsilon}(T_0)-V^{\varepsilon}_{\nu}(t)\|_{L^1}\\ &\quad +\|V^{\varepsilon}_{\nu}(t)-U^{\varepsilon}_{w,\nu}(t)\|_{L^1}\\[2mm]
&\leq L \|V^{\varepsilon}_{\nu}(T_0)-U^{\varepsilon}_{w,\nu}(T_0)\|_{L^1(\mathrm{R})}
      +\|S(t-T_0)V_{\nu}^{\varepsilon}(T_0)-V^{\varepsilon}_{\nu}(t)\|_{L^1}
     +\|V^{\varepsilon}_{\nu}(t)-U^{\varepsilon}_{w,\nu}(t)\|_{L^1}\\[2mm]
&\leq \|S(t-T_0)V_{\nu}^{\varepsilon}(T_0)-V^{\varepsilon}_{\nu}(t)\|_{L^1}
      + C\,\varepsilon^2\sum_{j\in N}\|\sigma^{(j)}_{\nu}(\varepsilon T_0, \cdot-\lambda_j^0T_0)\|_{L^\infty}
      \|\sigma^{(j)}_{\nu}(\varepsilon T_0, \cdot-\lambda_j^0T_0)\|_{L^1}
        \\
     &\quad +C\,\varepsilon^2\sum_{j\in N}\|\sigma^{(j)}_{\nu}(\varepsilon t, \cdot-\lambda_j^0t)\|_{L^\infty}
      \|\sigma^{(j)}_{\nu}(\varepsilon t, \cdot-\lambda_j^0t)\|_{L^1}
      +C\,\varepsilon^2\sum_{j\in L}|\mbox{supp}_x\sigma_{\nu}^{(j)}|\\
&\leq \|S(t-T_0)V_{\nu}^{\varepsilon}(T_0)-V^{\varepsilon}_{\nu}(t)\|_{L^1} + C\varepsilon^2.
\end{align*}
Then, passing the limit $\nu\rightarrow\infty$ and using Lemma \ref{prop:V}, we have
$$
\|S(t-T_0)U^{\varepsilon}_w(T_0)-U^{\varepsilon}_w(t)\|_{L^1}\leq C\,\varepsilon^2,
$$
which completes the proof.
\end{proof}

With Lemmas \ref{lem:GN Riemann compact support}--\ref{prop:CS},
we can now complete the proof of Theorem \ref{thm:compacted initial data}.

\begin{proof}[$\mathbf{Proof~of~Theorem~\ref{thm:compacted initial data}}$] Notice that
\begin{align*}
    &\|S(t)U^{\varepsilon}_w(0)-U^{\varepsilon}_w(t)\|_{L^1}\\[2mm]
&\leq \|S(t)U^{\varepsilon}_w(0)-S(t-T_0)U^{\varepsilon}_w(T_0)\|_{L^1}
+\|S(t-T_0)U^{\varepsilon}_w(T_0)-U^{\varepsilon}_w(t)\|_{L^1}\\[2mm]
&=  I_1+I_2.
\end{align*}
Using \eqref{coreissue} in \S \ref{sec:main thm}, we have
\begin{align*}
I_1&\leq \|S(t-T_0)S(T_0)U^{\varepsilon}_w(0)-S(t-T_0)U^{\varepsilon}_w(T_0)\|_{L^1}\\[2mm]
   &\leq L\|S(T_0)U^{\varepsilon}_w(0)-U^{\varepsilon}_w(T_0)\|_{L^1}\\[2mm]
   &\leq C\,T_0\varepsilon^2.
\end{align*}
By Lemma \ref{prop:CS},
$$
I_2\leq C\, \varepsilon^2.
$$
Therefore, we conclude
$$
\|U^{\varepsilon}(t,\cdot)-U^{\varepsilon}_w(t,\cdot)\|_{L^1}\le C\, \varepsilon^2.
$$
This completes the proof.
\end{proof}

\section{Proof of Theorem \ref{thm:noncompact case}.}\label{sec:noncompacted case}

Following the approach developed in \S 3--\S 5, we can extend the result even
for $BV\cap L^1$ initial data, Theorem \ref{thm:noncompact case}.

We first define the error terms $E_\nu(t,x)$ for each time $t$.
For any jump point $x$ of $\sigma^{(k)}_{\nu}(\varepsilon t, x-\lambda_k^0t)$,
the jump of $E_{\nu}$ at $x$ is
\begin{eqnarray}
E_{\nu}(t,x+)-E_{\nu}(t,x-)&=&\sum_{j\neq k}\sigma_{\nu}^{(j)}(\varepsilon t,x-\lambda_j^0t)(\sigma_{\nu+}^{(k)}-\sigma_{\nu-}^{(k)})(\rr_j^0\cdot\nabla)\rr_k^0\nonumber\\
&& +\frac{1}{2}\big((\sigma^{(k)}_{\nu+})^2-(\sigma^{(k)}_{\nu-})^2\big)(\rr_k^0\cdot\nabla)\rr_k^0.
\label{6.1a}
\end{eqnarray}

From the properties of $\sigma^{(j)}(\varepsilon t,x-\lambda_j^0 t)$,
we know that $E_{\nu}(t,\cdot)\in L^1(\mathbb{R};\mathbb{R}^n)\cap BV(\mathbb{R};\mathbb{R}^n)$
and, near the jump point $(t,x)$, the speed of discontinuity of $E_{\nu}(t,x)$ is the same as the
one of $\sigma_{\nu}^{(k)}(\varepsilon t, x-\lambda_k^0t)$.
In order to use the standard Riemann semigroup,
we need to show that the error terms $E_{\nu}(t,\cdot)$ are Lipschitz-continuous with respect to $t$.
Indeed, we have
\begin{align*}
&\|E_{\nu}(t,\cdot)-E_{\nu}(s,\cdot)\|_{L^1}\\
&=O(1)||\sigma^{(j)}_{\nu}||_{\infty}(t-s)\cdot[\mbox{maximum speed}]
\cdot[\mbox{total strength of all wave-fronts of }\sigma_{\nu}^{j}(\varepsilon s, \cdot)]\\
&\leq C(t-s)
\end{align*}
for some constant $C$ independent on $\nu$.
Then, as before, we modify expansion \eqref{og expansion}
by the following corresponding approximation:
\begin{eqnarray}
V^{\varepsilon}_{\nu}=U^0+\varepsilon\sum_{j}\sigma^{(j)}_{\nu}(\varepsilon t, x-\lambda_j^0t)\rr^0_j
+E_{\nu}(t,x) \varepsilon^2.
\end{eqnarray}

Then, as the bootstraps in \S \ref{sec:compact support}, we first have the following lemma.

\begin{lem}\label{lem:Riemann noncompact}
Suppose that the assumptions in Theorem {\rm \ref{thm:noncompact case}} hold and, at point $x$,
$\sigma^{(k)}(\varepsilon t,x-\lambda t)$ has a jump. Let
$$
V_{\nu\pm}^{\varepsilon}:=U^0+\varepsilon\sigma^{(k)}_{\nu\pm}\rr^0_{k} +E_{\nu\pm}(t,x) \varepsilon^2
$$
with $\Phi(V_{\nu-}^{\varepsilon};\beta_1,\cdot\,\cdot\,\cdot,\beta_n)=V_{\nu+}^{\varepsilon}$. Then
\begin{equation}
\beta_j=\sigma\varepsilon\delta_{kj}
+O(1)\epsilon^3\sum_{i}\big((\sigma_{\nu-}^{(i)})^2\sigma+\sigma_{\nu-}^{(i)}\sigma^2+\sigma^3\big),
\end{equation}
where $\sigma=\sigma_{\nu+}^{(k)}-\sigma_{\nu-}^{(k)}$, and $\delta_{kj}$ is the Kronecker delta.
\end{lem}

\begin{proof}
Let $\theta=\varepsilon\sigma$. Then
$$
V^{\varepsilon}_{\nu+}-V^{\varepsilon}_{\nu-}=\theta \rr^0_{k}+\frac{\theta^2}{2}(\rr_k^0\cdot\nabla)\rr_k^0
+\varepsilon\theta\sum_{j}\sigma_{\nu-}^{(j)}(\rr_j^0\cdot\nabla)\rr_k^0.
$$
Thus, we have the following equation:
\begin{equation}\label{equ:star1}
\Phi(V^{\varepsilon}_{\nu-};\,\beta_1,\,\cdot\,\cdot\,\cdot,\,\beta_n)
=V^{\varepsilon}_{\nu-}+\theta \rr_k^0+ \frac{\theta^2}{2}(\rr_k^0\cdot\nabla)\rr_k^0+\varepsilon\theta\sum_{j}\sigma_{\nu-}^{(j)}(\rr_j^0\cdot\nabla)\rr_k^0.
\end{equation}
Clearly, by the implicit function theorem, there exists a unique solution $(\beta_1, \cdots, \beta_n)(\theta)$
of \eqref{equ:star1} such that
$\beta_j(0)=0, j=1, \cdots, n$.

Differentiating \eqref{equ:star1} with respect to $\theta$ and letting $\theta=0$, we have
\begin{eqnarray}
&&\sum_{j}\frac{\partial\Phi(V^{\varepsilon}_{\nu-};\,\beta_1,\,\cdot\,\cdot\,\cdot,\,\beta_n)}{\partial\beta_j} \frac{\partial\beta_j}{\partial\theta}\Big|_{\theta=0} \nonumber\\
&&=\rr_k^0+\sum_j\varepsilon\sigma_{\nu-}^{(j)}(\rr^{0}_j\cdot\nabla)\rr_k^0
+\theta (\rr_k^0\cdot\nabla)\rr_k^0\big|_{\theta=0}.
\label{6.5a}
\end{eqnarray}
Notice that
$\rr_k^0+\varepsilon\sum_j\sigma_{\nu-}^{(j)}(\rr^{0}_j\cdot\nabla)\rr_k^0
=\rr_j(V_{\nu-}^{\varepsilon}) +O(1) \varepsilon \sum_j|\sigma_{\nu-}^{(j)}|$.
Then we obtain from \eqref{6.5a} that
\begin{align*}
\sum_{j}\frac{\partial\Phi}{\partial\beta_j} \frac{\partial\beta_j}{\partial\theta}\Big|_{\theta=0} =\sum_j\rr_j(V_{\nu-}^{\varepsilon})\frac{\partial\beta_j}{\partial\theta}\big|_{\theta=0}
=\rr_k(V_-^{\varepsilon})+O(1)\varepsilon^2\sum_j(\sigma_{\nu-}^{(j)})^2,
\end{align*}
which yields
$$
\frac{\partial\beta_j}{\partial\theta}\Big|_{\theta=0}=\delta_{kj}+O(1)\varepsilon^2\sum_{i}(\sigma_{\nu-}^{(i)})^2.
$$
Next, takeing the twice derivatives on \eqref{equ:star1} with respect to $\theta$ and letting $\theta=0$ yield
$$
\sum_{i,j}\frac{\partial^2\Phi}{\partial\beta_i\partial\beta_j} \frac{\partial\beta_i}{\partial\theta}\frac{\partial\beta_j}{\partial\theta}\Big|_{\theta=0} +\sum_j\frac{\partial\Phi}{\partial\beta_j}\frac{\partial^2\beta}{\partial\theta^2}\Big|_{\theta=0}
=(\rr_k\cdot\nabla)\rr_k^0.
$$
Thus, we have
$$
(\rr_k\cdot\nabla) \rr_k(V_{\nu-}^{\varepsilon})+O(1)\varepsilon^2\sum_i(\sigma_{\nu-}^{(i)})^2 +\sum_j\rr_j(V_{\nu-}^{\varepsilon})\frac{\partial^2\beta_j}{\partial\theta^2}\Big|_{\theta=0}
=(\rr_k\cdot\nabla)\rr_k^0.
$$
Noticing that $(\rr_k(V_{\nu-}^{\varepsilon})\cdot\nabla)\rr_k(V_{\nu-}^{\varepsilon})
=(\rr_k^0\cdot\nabla)\rr_k^0+O(1)\varepsilon\sigma_{\nu-}^{(k)}$, we have
$$
\frac{\partial^2\beta_j}{\partial\theta^2}\Big|_{\theta=0}=O(1)\varepsilon\sum_i\sigma_{\nu-}^{(i)}.
$$
With all of these formulas above, by the Taylor expansion, we finally have
\begin{align*}
\beta_j(\theta)
=&\beta_j(0)+\frac{\partial\beta_j}{\partial\theta}\big|_{\theta=0}\theta +\frac{\partial^2\beta_j}{\partial\theta^2}\big|_{\theta=0}\theta^2+O(1)\theta^3\\[2mm]
=&\sigma\varepsilon\delta_{kj}+O(1)\varepsilon^3\sum_i\big((\sigma_{\nu-}^{i})^2\sigma+\sigma_{\nu-}^{(i)}\sigma^2+\sigma^3\big).
\end{align*}
It completes the proof of this lemma.
\end{proof}

With this better estimate, as in \S \ref{sec:compact support}, we have

\begin{lem}\label{prop:Riemann noncompact}
Suppose that the assumptions in Theorem  {\rm \ref{thm:noncompact case}} hold, and
the jump strength of $\sigma^{(k)}_{\nu}$ is $\sigma_k$.
Then there exists $h_0>0$ such that, when $0< h\le h_0$, we have
\begin{eqnarray*}
&& \int_{x_0-\hat{\lambda}h<x<x_0+\hat{\lambda}h} |S_h(V_{\nu}^{\varepsilon}(t_0,x))-V^{\varepsilon}_{\nu}(t_0+h,x)|\mbox{d}x\\[2mm]
&&\leq C\sum_{\lambda_j^0\neq\lambda_k^0} |\sigma_{\nu-}^{(j)}||\sigma_k|h \varepsilon^2
 + C\big(\sum_j|\sigma_{\nu-}^{(j)}|\big)^2|\sigma_k|h \varepsilon^3
 + C 2^{-\nu}|\sigma_k|h \varepsilon^2.
\end{eqnarray*}
\end{lem}

\begin{proof} The proof is divided into three cases: the contact discontinuity case,
rarefaction case, and shock case.

For the rarefaction case, the proof is the same as the one in \S \ref{sec:compact support}, which needs use
the crucial property of the wave strength in the scheme that $\sigma\leq C 2^{-\nu}$.
Thus, it suffices to consider the other two cases.

Consider $\sigma_k<0$ for the shock.
By Lemma \ref{lem:Riemann noncompact},  $\beta_k<0$.
Denote $S_k$ the k-shock speed of $\beta_k$ in the $(t,x)$--coordinates and $V_k=\Phi(V_{\nu-}^{\varepsilon};\,\beta_1,\,\cdot\,\cdot\,\cdot,\beta_k,0,\cdot\,\cdot\,\cdot,0)$.
Then
\begin{eqnarray*}
S_k&=&\lambda_k(V_{k-1})+\frac{1}{2}\beta_k+O(1)|\beta_k|^2 \\[2mm]
&=&\lambda_k^0+\varepsilon\sigma_{\nu-}^{(k)}+\frac{\sigma_k\varepsilon}{2}
  +O(1)\varepsilon\sum_{j\neq k} \sigma_{\nu-}^{(j)} +O(1)\varepsilon^2(\sigma^2_k+\varepsilon|\sigma_k|).
\end{eqnarray*}
On the other hand, the shock speed of the Burgers equation of $\sigma^{(k)}_{\nu}$ in the $(t,x)$--coordinates is
$$
S_{Bk}=\lambda_k^0+\sigma_{\nu-}^{(k)}\varepsilon +\frac{1}{2}\sigma_k\varepsilon.
$$
Then
$$
S_k-S_{Bk}=O(1)\varepsilon\sum_{j\neq k} |\sigma_{\nu-}^{(j)}|+ O(1)(\sigma^2_k+|\sigma_k||\sigma_{\nu-}^{(k)}|)\varepsilon^2.
$$
Meanwhile, with the fact that $|V_{k-1}-V_n|+|V_k-V_0|=O(1)|\sigma_k|\varepsilon$,
we deduce the core estimate for this case:
$$
|S_k-S_{Bk}|\big(|V_{k-1}-V_{\nu+}^{\epsilon}|+|V_k-V_{\nu-}^{\varepsilon}|\big)
=O(1)\varepsilon^2\sum_{j\neq k} |\sigma_{\nu-}^{(j)}||\sigma_k|+O(1)(|\sigma_k|^3+|\sigma_k|^2|\sigma_{\nu-}^{(k)}|)\varepsilon^3.
$$

Consider the case of contact discontinuity whose corresponding characteristic field has constant multiplicity:
$\rr_j\cdot\nabla\lambda_k\equiv0$ if $\lambda_j=\lambda_k$.
Then the kth-shock speed $S_k$ of $\beta_k$ in the $(t,x)$--coordinates is
$$
S_k=\lambda_k(V_{k-1})=\lambda_k^0 +O(1)\varepsilon\sum_{\lambda_j^0\neq\lambda_k^0} \sigma_{\nu-}^{(j)}.
$$
On the other hand, the shock speed of the Burgers equation of $\sigma_{k}$ in the $(t,x)$--coordinates is
$$
S_{Bk}=\lambda_k^0.
$$
Then we can obtain the same core estimate for this case as for the shock case.
With these estimates in hand, following the bootstrap in \S \ref{sec:compact support} step by step,
we have
\begin{eqnarray*}
&& \int_{x_0-\hat{\lambda}h<x<x_0+\hat{\lambda}h}|S_h(V_{\nu}^{\varepsilon}(t_0,x))-V^{\varepsilon}_{\nu}(t_0+h,x)| \mbox{d} x \\[2mm]
&& \leq C \sum_{\lambda_j^0\neq\lambda_k^0} |\sigma_{\nu-}^{(j)}||\sigma_k|h \varepsilon^2+ C\,\big(\sum_{j}|\sigma_{\nu-}^{(j)}|\big)^2|\sigma_k|h\varepsilon^3.
\end{eqnarray*}

Finally, combining this inequality with the inequality from the rarefaction wave case,
we completes the proof.
\end{proof}

Then, summing up the above estimates, we have

\begin{lem}\label{lem:noncompactproperty}
Suppose that the assumptions in Theorem {\rm \ref{thm:noncompact case}} hold. Then, for $t>0$, we have
\begin{eqnarray*}
&&\|S(t)V^{\varepsilon}_{\nu}(0)-V^{\varepsilon}_{\nu}(t)\|_{L^1(\mathrm{R})}\\
&&\leq C\,(\varepsilon t+2^{-\nu}t)\varepsilon^2
+C\varepsilon^2 \int_{0}^{t}\sum_{k}(\sum_{x\in J_k(s)}\sum_{\lambda_j^0\neq\lambda_k^0} |\sigma^{(j)}_{\nu}(\varepsilon s,x-\lambda_js)||\sigma_k(x)|)\mbox{d}s,
\end{eqnarray*}
where $x\in J_k(s)$ means that, at the point $(\varepsilon s, x-\lambda_{k}s)$,
$\sigma^{(k)}_{\nu}$ has a jump with the corresponding jump strength $\sigma_k(x)$.
\end{lem}

This can be achieved by first using the standard semigroup error formula \eqref{semigroup estimate}, then letting $h$
small enough such that $S(h)V_{\nu}^{\epsilon}$ is the solution obtained by piecing together
the standard entropy solutions of the Riemann problems determined by the jumps of $V_{\nu}^{\epsilon}$, and
finally using Lemma \ref{prop:Riemann noncompact}.

Since the estimates for the auxiliary function have been established,
we are now at the stage to consider the estimates for the geometric optic expansion function $U_{w}^{\varepsilon}$.

\begin{lem}\label{prop:noncompact}
Suppose that the assumptions in Theorem {\rm \ref{thm:noncompact case}} hold. Let $I_m=(m,m+2)$. Then, for $t>0$, we have
\begin{eqnarray*}
&&\|S(t)U^{\varepsilon}_w(0)-U^{\varepsilon}_w(t)\|_{L^1(\mathrm{R})}\\
&&\leq C(\varepsilon^2+\varepsilon^3t)+C\varepsilon^2\int_{0}^{t}
\sum_{m\in\mathbb{Z}}\sum_{\lambda_j^0\neq\lambda_k^0}\|\sigma^{(j)}\|_{L^{\infty}(I_m)}
TV_{x\in I_m}\big(\sigma^{(k)}\big)\mbox{d}s.
\end{eqnarray*}
\end{lem}

\begin{proof} This is a simple corollary of Lemma \ref{lem:noncompactproperty}. Notice that
\begin{align*}
U^{\varepsilon}_{w,\nu}(t,x)=V^{\varepsilon}_{\nu}(t,x)-E_{\nu}(t,x)\varepsilon^2,
\end{align*}
and that $E_{\nu}(t,\cdot )\in BV(\mathbb{R};\mathbb{R}^n)\cap L^1(\mathbb{R};\mathbb{R}^n)$.
Thus, by Lemma \ref{prop:stability}, we have
\begin{eqnarray*}
&&\|S(t)U^{\varepsilon}_{w,\nu}(0)-U^{\varepsilon}_{w,\nu}(t)\|_{L^1(\mathrm{R})}\\[2mm]
&&\leq \|S(t)U^{\varepsilon}_{w,\nu}(0)-S(t)V_{\nu}^{\varepsilon}(0)\|_{L^1(\mathrm{R})}
     +\|S(t)V_{\nu}^{\varepsilon}(0)-V^{\varepsilon}_{\nu}(t)\|_{L^1(\mathrm{R})}\\[1.5mm]
&&\quad +\|V^{\varepsilon}_{\nu}(t)-U^{\varepsilon}_{w,\nu}(t)\|_{L^1(\mathrm{R})}\\[2mm]
&&\leq L \|V^{\varepsilon}_{\nu}(0)-U^{\varepsilon}_{w,\nu}(0)\|_{L^1(\mathrm{R})}
     +C\,(\varepsilon t+2^{-\nu}t)\varepsilon^2
     +\|V^{\varepsilon}_{\nu}(t) -U^{\varepsilon}_{w,\nu}(t)\|_{L^1(\mathrm{R})}\\
&&\quad +C\varepsilon^2 \int_{0}^{t}\sum_{k}\big(\sum_{x\in J_k(s)}\sum_{\lambda_j^0\neq\lambda_k^0} |\sigma^{(j)}_{\nu}(\varepsilon s,x-\lambda_js)||\sigma_k(x)|\big)\mbox{d}s\\
&&\leq C\,(1+\varepsilon t+2^{-\nu}(t-T_0))\varepsilon^2 +C\varepsilon^2 \int_{0}^{t}\sum_{k}\big(\sum_{x\in J_k(s)}\sum_{\lambda_j^0\neq\lambda_k^0}|\sigma^{(j)}_{\nu}(\varepsilon s,x-\lambda_js)||\sigma_k(x)|\big)\mbox{d}s.
\end{eqnarray*}
Then, passing the limit $\nu\rightarrow\infty$, we have
$$
\|S(t)U^{\varepsilon}_w(0)-U^{\varepsilon}_w(t)\|_{L^1(\mathrm{R})}\leq C(\varepsilon^2+ \varepsilon^3t)+C\varepsilon^2 \int_{0}^{t}\sum_{m\in\mathbb{Z}}\sum_{\lambda_j^0\neq\lambda_k^0} \|\sigma^{(j)}\|_{L^{\infty}(I_m)}
TV_{x\in I_m}\big(\sigma^{(k)}\big)\mbox{d}s,
$$
which completes the proof.
\end{proof}

With Lemmas \ref{lem:Riemann noncompact}--\ref{prop:noncompact},
we now prove Theorem \ref{thm:noncompact case}.

\begin{proof}[$\mathbf{Proof~of~Theorem~\ref{thm:noncompact case}}$]  It suffices to prove
$$
\lim_{\varepsilon\rightarrow0}\sup_{0\leq t\leq T_0/\varepsilon}\varepsilon \int_{0}^{t}\sum_{m\in\mathbb{Z}} \sum_{\lambda_j^0\neq\lambda_k^0}\|\sigma^{(j)}\|_{L^{\infty}(I_m)}
TV_{x\in I_m}\big(\sigma^{(k)}\big)\mbox{d}s\rightarrow0,
$$
where $I_m=(m,m+2)$. Let $\tau=\varepsilon s$. Then
\begin{align*}
&\lim_{\varepsilon\rightarrow0}\sup_{0\leq t\leq T_0/\varepsilon}\varepsilon \int_{0}^{t}\sum_{m\in\mathbb{Z}} \sum_{\lambda_j^0\neq\lambda_k^0}\|\sigma^{(j)}\|_{L^{\infty}(I_m)}
TV_{x\in I_m}\big(\sigma^{(k)}\big)\mbox{d}s\\
&\leq \lim_{\varepsilon\rightarrow0}\varepsilon \int_{0}^{T_0/\varepsilon}\sum_{m\in\mathbb{Z}}\sum_{\lambda_j^0\neq\lambda_k^0}
\|\sigma^{(j)}\|_{L^{\infty}(I_m)}TV_{x\in I_m}\big(\sigma^{(k)}\big)\mbox{d}s\\
&=\lim_{\varepsilon\rightarrow0}\int_{0}^{T_0}\sum_{m\in\mathbb{Z}}\sum_{\lambda_j^0\neq\lambda_k^0}
\|\sigma^{(j)}(\tau,\cdot-\frac{\lambda_j^0}{\varepsilon}\tau)\|_{L^{\infty}(I_m)}
TV_{x\in I_m}\big(\sigma^{(k)}(\tau,\cdot-\frac{\lambda_k^0}{\varepsilon}\tau)\big)\mbox{d}\tau.
\end{align*}
Since $\sigma^{j}(\tau,y)\in BV\cap L^1$ and $\lambda_j^0\neq\lambda_k^0$, then,
for any given $\tau\in(0,T_0)$ and any $1\leq j\leq n$,
we have
$$
\lim_{y\rightarrow\pm\infty}\sigma^{(j)}(\tau,y)=0, \qquad
\lim_{y\rightarrow\infty}TV_{(-\infty,-y)\cup(y,\infty)}(\sigma^{(j)})=0.
$$
Then, for any $\epsilon_1>0$, there exists a constant $M_j>0$ such that,
for any $|y|>M_j$, $|\sigma^{(j)}(\tau,y)|\leq\epsilon_1$.
Next, let $\varepsilon>0$ small enough such that, for any $\lambda_j^0\neq\lambda_k^0$,
$$
TV_{(-M_j-2+\frac{\lambda_j^0-\lambda_k^0}{\varepsilon}\tau, -M_j+2+\frac{\lambda_j^0-\lambda_k^0}{\varepsilon}\tau)}\big(\sigma(\tau,\cdot)\big)\leq\epsilon_1.
$$
Thus, for any $\varepsilon>0$ small enough, set
$$
J(j,\tau,\varepsilon):=\{m\in\mathbb{Z}\,\, \big|\,\, I_{m}\subset(-\infty,-M_j+\frac{\lambda_j^0}{\varepsilon}\tau) \cup(M_j+\frac{\lambda_j^0}{\varepsilon}\tau,\infty)\}.
$$
Then we have
\begin{align*}
&\sum_{m\in\mathbb{Z}}\sum_{\lambda_j^0\neq\lambda_k^0}
\|\sigma^{(j)}(\tau,\cdot-\frac{\lambda_j^0}{\varepsilon}\tau)\|_{L^{\infty}(I_m)}
TV_{x\in I_m}(\sigma^{(k)}\big(\tau,\cdot-\frac{\lambda_k^0}{\varepsilon}\tau)\big)\\
&=\sum_{\lambda_j^0\neq\lambda_k^0}\sum_{m\in J(j,\tau,\varepsilon)} \|\sigma^{(j)}(\tau,\cdot-\frac{\lambda_j^0}{\varepsilon}\tau)\|_{L^{\infty}(I_m)}
TV_{x\in I_m}\big(\sigma^{(k)}(\tau,\cdot-\frac{\lambda_k^0}{\varepsilon}\tau)\big)\\
&\quad +\sum_{\lambda_j^0\neq\lambda_k^0}\sum_{m\in(J(j,\tau,\varepsilon))^c} \|\sigma^{(j)}(\tau,\cdot-\frac{\lambda_j^0}{\varepsilon}\tau)\|_{L^{\infty}(I_m)}
TV_{x\in I_m}\big(\sigma^{(k)}(\tau,\cdot-\frac{\lambda_k^0}{\varepsilon}\tau)\big)\\
&\leq 2\sum_{\lambda_j^0\neq\lambda_k^0} \|\sigma^{(j)}(\tau,\cdot)
\|_{L^{\infty}((-\infty,-M_j)\cup(M_j,\infty))}TV\big(\sigma^{(k)}(\tau,\cdot)\big)\\
&\quad +2\sum_{\lambda_j^0\neq\lambda_k^0} \|\sigma^{(j)}(\tau,\cdot)\|_{L^{\infty}}
TV_{(-M_j-2+\frac{\lambda_j^0-\lambda_k^0}{\varepsilon}\tau, -M_j+2+\frac{\lambda_j^0-\lambda_k^0}{\varepsilon}\tau)}
\big(\sigma^{(k)}(\tau,\cdot)\big)\\
&\leq 2nTV\big(\sigma^{(k)}(0,\cdot)\big)\,\epsilon_1+2\sum_{j}\|\sigma^{(j)}\|_{L^{\infty}}\,\epsilon_1.
\end{align*}
Thus, we obtain that, for any $\tau\in(0,T_0)$,
$$
\lim_{\varepsilon\rightarrow0}\sum_{m\in\mathbb{Z}}\sum_{\lambda_j^0\neq\lambda_k^0}
\|\sigma^{(j)}(\tau,\cdot-\frac{\lambda_j^0}{\varepsilon}\tau)\|_{L^{\infty}(I_m)}
TV_{x\in I_m}\big(\sigma^{(k)}(\tau,x-\frac{\lambda_k^0}{\varepsilon}\tau)\big)=0.
$$
Meanwhile, let
$$
G(\tau):=C(n)\max_{j}\|\sigma^{(j)}(\tau, \cdot)\|_{L^\infty} \max_{j} TV\big(\sigma^{(j)}(\tau,\cdot)\big)
\leq C(n)\|U^1(\cdot)\|_{L^\infty}TV\big(U^1(\cdot)\big),
$$
where the constant $C(n)$ depends only on the dimension $n$. Then
$$
\sum_{m\in\mathbb{Z}}\sum_{\lambda_j^0\neq\lambda_k^0}
\|\sigma^{(j)}(\tau,\cdot-\frac{\lambda_j^0}{\varepsilon}\tau)\|_{L^{\infty}(I_m)}
TV_{x\in I_m}\big(\sigma^{(k)}(\tau,x-\frac{\lambda_k^0}{\varepsilon}\tau)\big)
\leq G(\tau).
$$
Notice that
$$
\int_0^{T_0}G(\tau)\mbox{d}\tau\leq C(n)T_0\|U^1(\cdot)\|_{\infty}TV\big(U^1(\cdot)\big)<\infty.
$$
Then, by the dominant convergence theorem, we have
$$
\lim_{\varepsilon\rightarrow0}\int_{0}^{T_0}\sum_{m\in\mathbb{Z}}\sum_{\lambda_j^0\neq\lambda_k^0}| |\sigma^{(j)}(\tau,\cdot-\frac{\lambda_j^0}{\varepsilon}\tau)||_{L^{\infty}(I_m)}
TV_{x\in I_m}\big(\sigma^{(k)}(\tau, x-\frac{\lambda_k^0}{\varepsilon}\tau)\big)\mbox{d}\tau\rightarrow0.
$$
Therefore, we conclude
$$
\sup_{0\leq t\leq T_0/\varepsilon}\|U^{\varepsilon}(t,\cdot)-U^{\varepsilon}_w(t,\cdot)\|_{L^1}=o(\varepsilon)
\qquad\mbox{when}\, \varepsilon\rightarrow0.
$$
\end{proof}

\bigskip
\paragraph{Acknowledgements.}  The authors thank the referees for valuable suggestions and comments.
The research of
Gui-Qiang Chen was supported in part by the National Science
Foundation under Grant DMS-0807551, the UK EPSRC Science and Innovation
Award to the Oxford Centre for Nonlinear PDE (EP/E035027/1),
the NSFC under a joint project Grant 10728101, and
the Royal Society--Wolfson Research Merit Award (UK).
Wei Xiang was supported in part by the China Scholarship Council  No.
2008631071 while
visiting the University of Oxford and the Doctoral Program Foundation of the Ministry Education of China.
Yongqian Zhang was supported in part by NSFC Project 11031001, NSFC Project 11121101,
and the 111 Project B08018 (China).
\bigskip

\bibliographystyle{plain}

\end{document}